\documentclass[reqno,10pt]{amsart}
\usepackage{euscript}
\usepackage{amssymb}
\usepackage{amscd}
\usepackage{amsmath}
\usepackage{epic}
\usepackage{graphics}
\usepackage{epsfig}
\usepackage{color}
\usepackage{setspace}
\usepackage[backref]{hyperref}

\numberwithin{equation}{section}
\newtheoremstyle{my}{1.5em}{0.5em}{\em}{}{\sc}{.}{0.5em}{}
\newtheoremstyle{your}{1.5em}{0.5em}{}{}{\sc}{.}{0.5em}{}

\theoremstyle{my}

\theoremstyle{my}
\newtheorem{thm}{Theorem}[section]
\newtheorem{Theorem}[thm]{Theorem}
\newtheorem*{Theorem*}{Theorem}
\newtheorem{Corollary}[thm]{Corollary}

\newtheorem*{corollary*}{Corollary}
\newtheorem{Lemma}[thm]{Lemma}

\newtheorem{Proposition}[thm]{Proposition}

\newtheorem*{conjecture*}{Conjecture}

\newtheorem*{question*}{Question}

\newtheorem*{definitions*}{Definitions}

\newtheorem*{rem*}{Remark}

\newtheorem*{remark*}{Remark}

\newtheorem*{remarks*}{Remarks}
\newtheorem*{example*}{Example}
\newtheorem*{examples*}{Examples}

\newtheorem*{convention*}{Convention}
\newtheorem*{conventions*}{Conventions}
\newtheorem*{Note*}{Note}
\newtheorem*{exercise*}{Exercise}
\newtheorem*{bibliographical-note*}{Bibliographical note}

\theoremstyle{your}
\newtheorem{Remark}[thm]{Remark}
\newtheorem{Definition}[thm]{Definition}

\newcommand{\Acknowledgements}{{\em Acknowledgements.} }

\newcommand{\R}{\mathbb{R}}
\newcommand{\Z}{\mathbb{Z}}

\newcommand{\C}{\mathbb{C}}
\newcommand{\ev}{\operatorname{ev}}

\newcommand{\pa}{\partial}
\newcommand{\Ordo}{\mathcal{O}}

\newcommand{\bR}{\mathbb{R}}
\newcommand{\bZ}{\mathbb{Z}}
\newcommand{\bQ}{\mathbb{Q}}
\newcommand{\bC}{\mathbb{C}}

\newcommand{\cdbar}{\mathrm{\overline{\partial}}}

\newcommand{\id}{\mathrm{id}}
\newcommand{\ind}{\mathrm{ind}}
\newcommand{\re}{\mathrm{Re}}
\newcommand{\im}{\mathrm{Im}}
\newcommand{\rk}{\mathrm{rank}}
\renewcommand{\ker}{\mathrm{ker}}

\newcommand{\Hom}{\mathrm{Hom}}

\newcommand{\bd}{\mathrm{bd}}

\newcommand{\sblv}{\mathcal{H}}
\newcommand{\cfig}{\mathcal{X}}

\newcommand{\tcfig}{\mathcal{Y}}

\newcommand{\scrB}{\mathcal{B}}
\newcommand{\scrJ}{\mathcal{J}}

\newcommand{\MM}{\mathcal{M}}
\newcommand{\FF}{\mathcal{F}}
\newcommand{\BB}{\mathcal{B}}

\newcommand{\JJ}{\mathcal{J}}

\newcommand{\EE}{\mathcal{E}}
\newcommand{\XX}{\mathcal{X}}
\newcommand{\YY}{\mathcal{Y}}
\newcommand{\NN}{\mathcal{N}}

\newcommand{\DD}{\mathcal{D}}
\newcommand{\LL}{\mathcal{L}}

\newcommand{\diam}{\operatorname{diam}}

\newcommand{\Exp}{\operatorname{Exp}}
\newcommand{\Pre}{\operatorname{PG}}

\title{Exact Lagrangian immersions with one double point revisited}
\author{Tobias Ekholm}
\address{Department of mathematics Uppsala University, Box 480, 751 06 Uppsala, Sweden \newline
Institut Mittag-Leffler, Aurav. 17, 182 60 Djursholm, Sweden}
\email{tobias.ekholm\@@math.uu.se}
\author{Ivan Smith}
\address{Centre for Mathematical Sciences, University of Cambridge, Wilberforce Road, CB3 0WB, UK}
\email{is200\@@cam.ac.uk} 

\thanks{TE was partially supported by the G\"oran Gustafsson Foundation for Research in Natural Sciences and Medicine.  \\ IS was partially supported by European Research Council grant ERC-2007-StG-205349.}
\begin{document}
\thispagestyle{empty}
\setcounter{tocdepth}{1}

\begin{abstract}
We study exact Lagrangian immersions with one double point of a closed orientable manifold $K$ into $\C^{n}$. We prove that if the Maslov grading of the double point does not equal $1$ then $K$ is homotopy equivalent to the sphere, and if, in addition, the Lagrangian Gauss map of the immersion is stably homotopic to that of the Whitney immersion, then $K$ bounds a parallelizable $(n+1)$-manifold. The hypothesis on the Gauss map always holds when $n=2k$ or when $n=8k-1$. The argument studies a filling of $K$ obtained from solutions to perturbed Cauchy-Riemann equations with boundary on the image $f(K)$ of the immersion. This leads to a new and simplified proof of some of the main results of \cite{EkholmSmith}, which treated Lagrangian immersions in the case $n=2k$ by applying similar techniques to a Lagrange surgery of the immersion, as well as to an extension of these results to the odd-dimensional case. 
   \end{abstract}

\maketitle

\section{Introduction}

We consider Euclidean space $\bC^n$ with coordinates $(x_1+iy_1,\dots,x_n+iy_n)$ and with its standard exact symplectic structure $\omega = -d\theta$, where $\theta=\sum_{j=1}^{n}y_j\,dx_j$. We are concerned with exact Lagrangian immersions $f\colon K \to \bC^n$ of closed orientable manifolds $K$, i.e.~immersions for which $f^{\ast}\theta = dz$ for a smooth function $z\colon K\to\bR$.  Gromov \cite{Gromov} proved that $\bC^n$ contains no embedded exact Lagrangian submanifolds, so immersions with a single (transverse) double point have the simplest possible self-intersection.

This paper is a sequel to \cite{EkholmSmith}, which showed that if $n>4$ is even, if $\chi(K)\ne -2$, and if $K$ admits an exact Lagrangian immersion into $\C^{n}$ with only one double point, then $K$ is diffeomorphic to the standard $n$-sphere.  The proof of this theorem relied on constructing a parallelizable manifold $\scrB$ with $\partial \scrB = K$  from moduli spaces of Floer holomorphic disks with boundary in a Lagrange surgery of the image $f(K)$ of the immersion.  The corresponding Lagrange surgery is non-orientable when $n$ is odd, hence cannot bound any parallelizable manifold.  In this paper, we study moduli spaces of Floer holomorphic disks with boundary on the immersed manifold $f(K)$ whose boundary values admit a continuous lift into $K$. This turns out to be analytically somewhat simpler, and leads to a result valid in all dimensions.

To state the theorem, recall that the map $\tilde f=(f,z)\colon K\to\C^{n}\times\R$ is a Legendrian immersion (generically an embedding) into $\C^{n}\times\R$ with the contact form $dz-\theta$. The double points of $f$ correspond to Reeb chords of $\tilde f$, and as such have a mod $2$ Maslov grading $|\cdot|_2$, which lifts to a mod $2j$ grading $|\cdot|_{2j}$ if the Maslov index of $f$ equals $2j$. (Note that the Maslov index is even since $K$ is orientable). Here $|\cdot|_{0}=|\cdot|$ is to be understood as an integer grading. The Lagrangian immersion also determines a Gauss map $ K\to U_n/O_n$, taking  $p\in K$ to the tangent plane $df(T_pK)$ viewed as a Lagrangian subspace of $\C^{n}$. The \emph{stable Gauss map} $Gf$ is the composition of the Gauss map with the natural map $U_n/O_n\to U/O$. Finally, we recall the classical Whitney immersion in Definition \ref{def:Whitney}.

\begin{Theorem}\label{thm:main}
Let $K$ be an orientable $n$-manifold that admits an exact Lagrangian immersion $f$ into $\bC^n$ of Maslov index $2j$, with exactly one double point $a$. Suppose  $|a|_{2j}\ne 1$. Then $j=0$, $|a|=n$, and $K \simeq S^n$ is homotopy equivalent to the sphere. If in addition the stable Gauss map $Gf$ of $K$ is homotopic to the stable Gauss map of the Whitney immersion of $S^n$ then $K$ bounds a parallelizable manifold.
\end{Theorem}

The stable Gauss map $Gf$ of a Lagrangian immersion $f$ of a homotopy $n$-sphere lifts to a map into $U$, hence vanishes if $n$ is even.  If $n$ is odd it does not necessarily vanish; the relevant homotopy groups are   
\[
\pi_n(U/O)=\begin{cases}
0 &\text{ for }n=8k-1,\\
\Z &\text{ for }n=8k+1\text{ or }n=8k+5,\\
\Z_{2} &\text{ for }n=8k+3.
\end{cases}
\]
In particular, the homotopy condition in Theorem \ref{thm:main} always holds when $n=8k-1$.

The only even-dimensional homotopy sphere which bounds a parallelizable manifold is the standard sphere, by \cite{KM}.   If $n$ is even, the mod 2 Maslov grading of the  double point is determined by whether the Euler characteristic is $+ 2$ or $-2$.  Combining these facts, Theorem \ref{thm:main} in particular recovers the results on even dimensional spheres in \cite{EkholmSmith} as stated above. 

Theorem \ref{thm:main} should be contrasted with the  results of \cite{EEM}, which provide general constructions of exact Lagrangian immersions with few double points, either by appeal to an $h$-principle for loose Legendrian embeddings and their Lagrangian caps \cite{Murphy, EM} or via  higher-dimensional analogues of the explicit constructions of \cite{Sauvaget}.  In particular, special cases of the result of \cite{EEM} show that if $K$ is an orientable $n$-manifold with $TK\otimes\bC$ trivial, then
\begin{itemize}
\item if $n=2k$ and $\chi(K) = -2$, then $K$ admits an exact Lagrangian immersion with exactly one double point;
\item if $n=2k+1$ and $K$ is a homotopy $n$-sphere, then $K$ admits an exact Lagrangian immersion with exactly one double point.
\end{itemize}
Thus, the main result of \cite{EkholmSmith} is essentially sharp, and the Maslov index hypothesis is crucial for the rigidity results established in Theorem \ref{thm:main}.  

In the proof of Theorem \ref{thm:main} the homotopy type of $K$ is controlled by a version of linearized Legendrian contact homology adapted to an abstract covering space of $K$,  analogous to Damian's lifted Floer homology \cite{Damian}. For a more systematic development of that subject, see \cite{Eriksson_Ostman}.  The essential simplification in the subsequent construction of a bounding manifold for $K$, as compared to \cite{EkholmSmith}, comes from the fact that breaking of Floer holomorphic disks is concentrated at the double point, and therefore Gromov-Floer boundaries of relevant moduli spaces are ordinary rather than fibered products.

If $n=4k+1$ there is at most one exotic sphere that bounds a parallelizable manifold.  In dimensions $n=4k-1$ there are typically many such exotic spheres, distinguished by the possible signatures of parallelizable $4k$-manifolds which they bound.  This paper includes a brief discussion of the signatures of the bounding manifolds obtained by Floer theory, see Section \ref{Sec:signature}.
 
Finally, we point out that \cite{EkholmSmith} also gave applications to Arnold's ``nearby Lagrangian submanifold conjecture", proving that $T^*((S^{n-1} \times S^1)\#\Sigma)$ and $T^*(S^{n-1} \times S^1)$ are not symplectomorphic for a non-trivial homotopy $n$-sphere $\Sigma$ when $n=8k$. That result relies essentially on the appearance of the Lagrange surgery of $f(K)$, and does not seem accessible via the simplified argument of this paper.

The paper is organized as follows. In Section \ref{Sec:topconstr} we derive the homotopy theoretic conclusions of Theorem \ref{thm:main}. In Section \ref{Sec:modulispaces} we discuss moduli spaces of Floer disks and their breaking, and in Section \ref{Sec:parallel} we construct a filling manifold and prove it is parallelizable.  The key analytical result, Theorem \ref{Thm:C1boundary}, asserting that a certain compactified moduli space of Floer solutions has a $C^1$-smooth structure, is proved in detail in Appendix \ref{Appendix}. Whilst most of the results in that appendix have appeared elsewhere, we hope that gathering them together and providing a self-contained and relatively concise account of the appropriate gluing theorem will be of wider use and interest. 

To avoid special cases, we assume throughout the paper that the real dimension $n$ of the source of the Lagrangian immersion satisfies $n>2$. 
\vspace{1em}

\Acknowledgements The authors are grateful to the referee, whose detailed reading of the manuscript led to numerous expositional improvements.

\section{Topological constraints on exact immersions}\label{Sec:topconstr}

We begin with some background material on Legendrian contact homology. For details, see \cite{EES1,EES} and references therein.

Let $\Lambda \subset (\bC^n \times \bR, dz-\theta)$ be an orientable Legendrian submanifold.  There is an associated Floer-theoretic invariant $LH(\Lambda;k)$, the \emph{Legendrian contact homology} of $\Lambda$ with coefficients in the field $k$. This is a unital differential graded algebra generated by the Reeb chords of $\Lambda$, and with differential counting certain punctured holomorphic curves in $\bC^n$ with boundary on the immersed Lagrangian projection of $\Lambda$. In the situations studied in this paper, all Reeb chords have non-zero length, and hence the algebra generators can be canonically identified with the double points of the Lagrangian projection.  An \emph{augmentation} of $LH(\Lambda;k)$ is a $k$-linear dga-map $\epsilon\colon LH(\Lambda;k) \rightarrow k$, where the field is concentrated in degree $0$ and equipped with the trivial differential. Such a map leads to a linearization of the the dga as follows. The $k$-vector space underlying the chain complex $C=C^{\mathrm{lin}}(\Lambda;k)$ of the linearization is generated by the Reeb chords of $\Lambda$. To define the differential on $C$, consider the tensor algebra 
\[
\overline{C} \ =\ k \ \oplus \ C \ \oplus \ C\otimes C \ \oplus \ \cdots \ \oplus \ C^{\otimes m}  \ \oplus \ \cdots
\]
and view the dga differential $\pa$ as an element 
\[
\pa\in C^{\ast}\otimes \overline{C},\quad \pa =\sum_{\text{Reeb chords }a} a^{\ast}\otimes (\pa a),
\]
where $C^{\ast}$ is the dual of $C$ and where we view $(\pa a)$ as a linear combination of monomials.  There is a canonical identification between $C$ and its space of first-order variations.  With this in mind, the differential $d$ on $C$ is the element in $C^{\ast}\otimes C$ given by composing the first order variation $\delta\pa$ with the augmentation map on coefficients. More concretely, if
\[
\pa = \sum_{i;j_1,\dots,j_n} \kappa_{i;j_{1},\dots,j_{n}}\, a^{\ast}_i\otimes b_{j_1}\dots b_{j_n},\quad  \kappa_{i;j_{1},\dots,j_{n}}\in k
\]    
then 
\[
\delta\pa = \sum_{i;j_1,\dots,j_n} \kappa_{i;j_{1},\dots,j_{n}}\,a_i^{\ast}\otimes \left((\delta b_{j_1})b_{j_2}\dots b_{j_n}+\dots+b_{j_1}\dots b_{j_{n-1}}(\delta b_{j_n})\right)
\]
and 
\[
d = \sum_{i;j_1,\dots,j_n} \kappa_{i;j_{1},\dots,j_{n}}\, a_i^{\ast}\otimes \left(\epsilon(b_{j_2}\dots b_{j_n})\,b_{j_1}+\dots+\epsilon(b_{j_1}\dots b_{j_{n-1}})\,b_{j_n}\right)\in C^{\ast}\otimes C.
\]

The invariant $LH(\Lambda; k)$ is well-defined as a mod $2$ graded algebra for arbitrary $\Lambda$ if the characteristic of $k$ equals $2$; in this generality, linearizations should preserve the mod $2$ degree.  If $\Lambda$ admits a spin structure, then one can work with fields $k$ of arbitrary characteristic. If $\Lambda$ has vanishing Maslov class, then the mod $2$ grading can be refined to a $\bZ$-grading, and one can ask for graded linearizations. 

Suppose $\pi\colon \tilde{\Lambda} \rightarrow \Lambda$ is an abstract (not necessarily finite degree) connected covering space of $\Lambda$ and that $\epsilon\colon LH(\Lambda;k)\to k$ is an augmentation. We will use a variant chain complex $\tilde C^{\mathrm{lin}}(\tilde{\Lambda}, \pi; k)$ which takes account of the cover, and which is analogous to the ``lifted Floer homology'' of Damian \cite{Damian}. To define it we first pick a base point $x\in \Lambda$ and pick ``capping paths" connecting all Reeb chord endpoints to $x$. 
Then, as in \cite{EES}, consider the Legendrian $\Lambda \cup \Lambda'$ comprising two copies of $\Lambda$, the original one and one copy $\Lambda'$ displaced by a large distance in the Reeb direction. This is a Bott-degenerate situation: through each point of $\Lambda$ there is a Reeb chord ending on $\Lambda'$. We perturb by fixing a sufficiently $C^2$-small Morse function on $\Lambda$ and replace $\Lambda'$ by the graph of the $1$-jet of this function, in a suitable Weinstein tubular neighborhood $\approx J^{1}(\Lambda)$ of $\Lambda'$. We keep the notation $\Lambda'$ for this perturbed component.

The chain complex $\tilde C= \tilde C^{\mathrm{lin}}(\tilde\Lambda,\pi;k)$ is generated by the preimages under $\pi$ of all endpoints of mixed Reeb chords, i.e.~chords that start on $\Lambda$ and end on $\Lambda'$. Let $\tilde c$ denote such a preimage. The differential $\tilde d \tilde c$ is defined as follows. Let $c$ be the mixed Reeb chord corresponding to $\tilde c$ and consider a rigid holomorphic disk with boundary on $\Lambda\cup\Lambda'$, with positive puncture at $c$, and with exactly one mixed negative puncture. Let the augmentation $\epsilon$ act on all pure negative punctures. We say that such a disk is \emph{contributing} if $\epsilon$ takes non-zero values at all pure Reeb chords along its boundary, otherwise we say it is \emph{non-contributing}. For contributing disks we add the capping paths at all pure Reeb chord endpoints in $\Lambda$. In this way, any contributing disk gives a path in $\Lambda$; given the start point $\tilde c\in \tilde\Lambda$ that path determines a unique path in $\tilde\Lambda$, which ends at a point $\tilde b$ which is one of the preimages of the endpoint of the mixed negative puncture in the disk. We define the \emph{coefficient} of any contributing disk to be the product of the values of $\epsilon$ on its pure chords. Finally, the differential $\tilde{d}\tilde{c}$ is defined as the sum of $\tilde{b}$ over all such rigid disks, weighted by the corresponding coefficient. 

\begin{Remark}
We point out that even when the covering degree and chain groups are infinite-rank, the differential of any given generator is well-defined (i.e.~involves only finitely many terms), since the set of chords of action bounded below by some fixed value is always finite, compare to  \cite[Lemma 2.1]{Abouzaid:htpy}. Note that the chain complex $\tilde C$ depends on the choice of capping paths.  
\end{Remark}

For the purposes of this paper, we will need only a number of elementary properties of this theory which we discuss next. We refer to \cite{Eriksson_Ostman} for a more systematic study of the construction. Note first that the mixed Reeb chords of $\Lambda\cup\Lambda'$ are of three kinds: short, Morse, and long. Here short chords correspond to chords that start near the upper endpoint of a Reeb chord of $\Lambda$ and end near the lower end point of a Reeb chord of $\Lambda'$, Morse chords correspond to critical points of the small perturbing function, and long chords start near the lower end point of a Reeb chord in $\Lambda$ and end near the upper endpoint of a Reeb chord of $\Lambda'$. Let $\tilde{S}$, $\tilde C_{\mathrm{Morse}}$, and $\tilde{L}$ denote the subspaces of $\tilde C$ generated by preimages of short, Morse, and long chords, respectively. Since any holomorphic disk has positive area, one finds that   
\[
\tilde{S}\quad\text{ and }\quad \tilde{S}\oplus\tilde{C}_{\textrm{Morse}}
\]
are both sub-complexes of $\tilde C$.

Assume now that $\Lambda$ has vanishing Maslov class. There is then a grading on $\tilde C$ that is well-defined up to an overall integer shift. In order to fix the grading we fix the Maslov potential difference between $\Lambda'$ and $\Lambda$ to be $1$. More prosaically, the grading of a chord is defined as a certain Maslov index of its associated  capping path. For mixed chords, such capping paths also include a path going between the components $\Lambda, \Lambda'$, and we take this path to have Maslov index $1$, see \cite[Section 2.1]{BEE}. With this choice,  elements $\tilde c\in \tilde{C}_{\textrm{Morse}}$ are graded by the Morse index of the corresponding critical point, whilst elements $\tilde{s}\in\tilde{S}$ and $\tilde l\in\tilde L$ are graded by shifts of their gradings downstairs: $|\tilde s|=|s|-(n+1)$ and $|\tilde l|=|l|+1$, where $s$ and $l$ are the Reeb chords of $\Lambda$ corresponding to $\tilde s$ and $\tilde l$, respectively. 

Finally, we will use the fact that the homology of $\tilde C$ is invariant under Hamiltonian isotopies of $\C^{n}$ and their induced actions on $\Lambda$; this is straightforward to derive from corresponding results for the two-copy Legendrian contact homology established in \cite[Section 2]{EESa}.

We now return to our particular situation. Let $f\colon K \to \bC^n$ be an exact Lagrangian immersion of an orientable manifold with one double point, and let $\tilde f\colon K\to\C^{n}\times\R$ be the associated Legendrian embedding. Let $a$ denote the unique double point, which is the canonical  generator of $LH(\tilde f(K);k)$.  

\begin{Lemma}\label{lemma:hmtpysphere}
One of the following holds:
\begin{enumerate}
\item  $LH(\tilde f(K);k)$ admits no mod $2$ linearization. In this case, $|a|_2 = 1$, and if $LH(\tilde f(K);k)$ admits an integer grading then $|a| = 1$. 
\item $LH(\tilde f(K); \bZ_2)$ admits a (mod $2$ graded) linearization.  In this case, the Maslov index of $f$ vanishes, $LH(\tilde f(K);k)$ admits an  integer grading, and $|a| = n$. Furthermore, $K$ is a homotopy sphere.
\end{enumerate}
\end{Lemma}

\begin{proof}
For the first case, any dga generated by a single element $a$ can only fail to be linearizable if $|a|_{2j} = 1$, for any value $j$ for which the dga admits a mod $2j$ grading.

Suppose then $LH(\tilde f(K);\bZ_2)$ admits a linearization. Since $\tilde f(K)$ is displaceable by Hamiltonian isotopy, the double point bounds of \cite{EES,EESa} apply and show that 
\begin{equation} \label{doublept}
\rk_{\bZ_2} (H^*(K;\bZ_2)) \leq 2
\end{equation}
which implies that $K$ is a $\bZ_2$-homology sphere.  Therefore $H^2(K;\bZ_2)=0$, so $K$ is spin and $LH(\tilde f(K);k)$ can be defined with coefficients in a characteristic zero field $k$. The double point bound \eqref{doublept} over $k$ then implies that $K$ is a rational homology sphere, hence has vanishing Maslov class, and \cite[Theorem 5.5]{EESa} then implies $|a|=n$. Finally, working with finite field co-efficients, one concludes that $K$ is a $\bZ$-homology sphere.  

It remains to prove that $\pi_1(K) = 1$. Consider a connected covering space $\pi\colon \tilde K\to K$ of $K$.  Then the complex $\tilde{C}^{\mathrm{lin}}(\tilde K,\pi;k)$ discussed above, has the form
\[
\tilde S\oplus \tilde C_{\mathrm{Morse}} \oplus \tilde L,
\]
where all elements in $\tilde S$ have grading $-1$ and all elements in $\tilde L$ grading $n+1$. Hamiltonian displaceability of $f(K)$ implies that the total homology of the complex vanishes. Since $\tilde{K}$ is connected, this in turn implies $|\tilde S|=1$.  That implies the covering $\pi$ has degree 1, and since the covering was arbitrary, that implies that $\pi_1(K) = 1$, as required.
\end{proof}

We should point out that one can also prove Lemma \ref{lemma:hmtpysphere} by applying the results of \cite{Damian} to the Lagrange surgery of $f(K)$, but that argument is somewhat more roundabout and involves more heavy algebraic machinery.

\section{Moduli spaces and breaking}\label{Sec:modulispaces}
 

\subsection{Displacing Hamiltonian and the Floer equation}\label{sec:dispham}
For the rest of the paper, we fix a Lagrangian immersion $f\colon\Sigma\to\bC^n$ of a homotopy $n$-sphere $\Sigma$ with exactly one transverse double point $a$. The double point has Maslov grading $|a|=n$.  
As above we write $\tilde f\colon \Sigma\to\C^{n}\times\R$ for the Legendrian lift of $f$, and we write $a^{\pm}\in \Sigma$ for the upper and lower endpoints of the Reeb chord corresponding to $a$.

Following \cite{Oh}, consider a compactly supported time-dependent Hamiltonian function
$H\colon \C^n \times [0,1] \to \R$ with the following properties:
\begin{itemize}
\item there exists $\epsilon_0>0$ such that $H_t$ is constant for $t\in [0,\epsilon_0]\cup[1-\epsilon_0,1]$;
\item the time $1$ flow $\phi^1$ of the Hamiltonian vector field $X_{H}$ of $H$ displaces $f(\Sigma)$ from itself: $\phi^1 (f(\Sigma)) \cap f(\Sigma) = \varnothing$.
\end{itemize}

We can suppose that $f(\Sigma)$ is real analytic \cite[Lemma 4.1]{EkholmSmith}, that the double point of $f$ lies at $0\in \bC^n$, and that the branches of $f(\Sigma)$ near the double point are flat and co-incide with $\bR^n \cup i\bR^n$ in a sufficiently small ball around the origin.

We fix a smooth family of smooth functions
\[
\alpha_r \colon \bR \longrightarrow [0,1] \quad \textrm{for} \ r \in [0,\infty)
\]
satisfying the conditions
\begin{itemize}
\item $\alpha_r(s) \ = \begin{cases} 0 \quad  |s| > r+1 \\ 1 \quad |s| \leq r
\end{cases}$
\item $\alpha_r'(s)$ is \,  $\begin{cases} < 0 \quad s\in (r, r+1) \\ >0 \quad s\in (-r-1, -r) \end{cases}$
\item $\alpha_r(s) = r\alpha_1(s)$ when $s\in [0,1]$, so in particular $\alpha_0 \equiv 0$.
\end{itemize}
If $D$ denotes the unit disk in $\C$ then there is a unique conformal map 
\begin{equation} \label{Eqn:xi}
\eta \colon \R\times[0,1]\to D \quad  \text{with} \quad \eta(\pm \infty)=\pm 1 \ \text{and} \ \eta(0)=-i,
\end{equation}
with inverse 
\begin{equation}\label{Eqn:xi-1}
\eta^{-1}=s+it\colon D\to\R\times[0,1].
\end{equation}
Let $\nu: [0,\infty) \rightarrow [0,1]$ be a non-decreasing smooth function which vanishes near $0$ and is identically $1$ for $r\geq 1$. 
Write $\gamma_r$ for the $1$-form on $D$ such that 
\[
\eta^{\ast}\gamma_{r}=\nu(r)\alpha_{r}(s)\,dt.
\]  

Fix a small $\delta>0$ and let $\JJ_\Sigma$ denote the space of almost complex structures on $\C^{n}$ which tame the standard symplectic form and which agree with the standard structure $J_0$ in a $\delta$-neighborhood of $f(\Sigma)$. Fix a finite set of points $\mathbf{z}=\{z_1,\dots,z_m\}\in\pa D$ subdivided into two subsets $\mathbf{z}=\mathbf{z}_{+}\cup\mathbf{z}_{-}$ and write
\[
D_{\mathbf{z}}=D-\mathbf{z}.
\]
Furthermore, for $p,q\ge 0$ write $D_{p,q}$ for $D_{\mathbf{z}}$ for an unspecified set $\mathbf{z}=\mathbf{z}_{+}\cup\mathbf{z}_{-}$ where there are $p$ boundary punctures in $\mathbf{z}_{+}$ and $q$ in $\mathbf{z}_{-}$. Note that $\partial D_{p,q}$ is a collection of $p+q$ open arcs. 
 For $r \in [0,\infty)$ and $J \in \JJ_\Sigma$ we will consider Floer disks with boundary on $\tilde f(\Sigma)$ with $p$ positive and $q$ negative boundary punctures at the double point $a$. Such a disk consists of the following:
\begin{itemize}
\item We have a map $u\colon (D_{p,q},\pa D_{p,q})\to (\C^{n},f(\Sigma))$ for which $u|_{\pa D_{p,q}}$ admits a continuous lift $\tilde u\colon \pa D_{p,q}\to \tilde f(\Sigma)$ with the property that for each positive boundary puncture $\xi\in\mathbf{z}_{+}$, $\lim_{\zeta\to \xi\pm}\tilde u(\zeta)=a^{\mp}$ and for each negative boundary puncture $\eta\in\mathbf{z}_{-}$, $\lim_{\zeta\to \eta\pm}\tilde u(\zeta)=a^{\pm}$. Here $\lim_{\zeta\to z+}$ and $\lim_{\zeta\to z-}$ denote the limits as $\zeta\to z$ in the counterclockwise and clockwise directions, respectively. 
\item For some $r(u) = r\in [0,\infty)$, the map $u$ solves the perturbed Cauchy-Riemann equation
\begin{equation} \label{Eqn:CR-1}
(du + \gamma_r \otimes X_H)^{0,1} = 0,
\end{equation}
where $A^{0,1}$ denotes the $(J, i)$ complex anti-linear part of the linear map $A\colon TD\to T\C^{n}$.
\end{itemize}

Let $\ell(a)$ denote the length of the Reeb chord corresponding to the double point $a$.

\begin{Lemma}\label{Lemma:energy}
There is a constant $C>0$ such that for every  Floer disk $u$ with $p$ positive and $q$ negative punctures, the following holds:
\[
\|du\|_{L^{2}}^{2}\le C
\left(
(p-q)\ell(a)+\int_{0}^{1} \left(\max H_{t}-\min H_{t}\right)dt
\right).
\]
\end{Lemma}

\begin{proof}
This follows from the first two calculations in \cite[Proof of Lemma 2.2]{Oh}, together   with Stokes' theorem which implies that  $\int_{D_{p,q}}u^{\ast}\omega=(p-q)\ell(a)$. (The value of the constant $C>0$ is related to the taming condition for $J$ and $\omega$.)
\end{proof}

If $r=0$, the Hamiltonian term in \eqref{Eqn:CR-1} vanishes identically, and Floer holomorphic disks are exactly holomorphic disks, i.e.~maps satisfying the unperturbed Cauchy-Riemann equations.   

\begin{Corollary} \label{Cor:Constants}
Suppose $r=0$. 
The space of solutions of \eqref{Eqn:CR-1} with no boundary punctures is canonically diffeomorphic to $\Sigma$, viewed as a space of constant maps.  Every non-constant solution to \eqref{Eqn:CR-1} has at least one positive puncture.
\end{Corollary}

\begin{proof}
For the first statement, Lemma \ref{Lemma:energy} implies that the derivative of any solution is zero, hence the solutions are all constant maps.  The fact that the  moduli space of constant maps is transversely cut out and diffeomorphic to $\Sigma$  is standard, and follows from the surjectivity of the ordinary $\bar\pa$-operator for maps $D\to\C^{n}$ with constant boundary condition $\R^{n}$. 
 The second part is immediate from Lemma \ref{Lemma:energy}.
\end{proof}

\subsection{Linearization and transversality}\label{Sec:transversality}
The linearization of \eqref{Eqn:CR-1} is a $C^1$-Fredholm operator which depends smoothly on the parameters $r$ and $\mathbf{z}$. To make this claim precise, we begin by defining appropriate configuration spaces of maps,  following \cite{EES}.  We briefly recall the set-up; these configuration spaces are re-introduced in more detail in the Appendix, Section \ref{Ssec:functional-analytic}.


The $\R$-coordinate of  the Legendrian lift $\tilde f$ of $f$ is constant in a neighborhood of the double point. We assume that the value at the lower branch is $0$, and let $\ell$ denote its value at the upper branch. Given a subdivision $\mathbf{z}=\mathbf{z}_{+}\cup\mathbf{z}_{-}$,  fix a smooth function $h_{0}\colon\pa D_{\mathbf{z}}\to\R$ which is locally constant near each puncture, and which jumps up from $0$ to $\ell$ at $z_j \in \mathbf{z}_+$ and down from $\ell$ to $0$ at $z_j\in\mathbf{z}_{-}$, in each case as one moves counterclockwise around $\pa D_{\mathbf{z}}$.

Fix a metric on $D_{\mathbf{z}}$ which agrees with the standard metric in half-strip coordinates near each puncture. Let $\mathcal{H}(D_{\mathbf{z}}, \bC^n)$ denote the Sobolev space of maps having two derivatives in $L^2$. 
We define $\XX_{\mathbf{z}}\subset \sblv^{2}(D_{\mathbf{z}};\C^{n})\times \sblv^{\frac32}(\pa D_{\mathbf{z}};\R)$ to be the subset of pairs $(u,h)$, where $u\colon D\to \C^{n}$ and $h\colon\pa D\to\R$ satisfy the following conditions:
\begin{enumerate}
\item $(u|_{\pa D},h+h_0)\colon \pa D\to \C^{n}\times\R$ takes values in the Legendrian lift $\tilde f(\Sigma)$ of $f\colon \Sigma\to\C^{n}$; 
\item $du^{0,1}|_{\pa D}=0$.
\end{enumerate}
Allowing $\mathbf{z}=\mathbf{z}_{+}\cup\mathbf{z}_{-}$ to vary, the configuration space $\XX_{p,q}$ fibers over the space $Z_{p,q} \subset (\partial D)^{p+q}$ of possible configurations of $p$ positive and $q$ negative boundary punctures, with fiber  $\XX_{\mathbf{z}}$ at a given set of punctures $\mathbf{z}$.  Furthermore, let $\YY_{\mathbf{z}} \to \XX_{\mathbf{z}}$ denote the bundle with fiber over $u\in\XX_{\mathbf{z}}$ equal to the closed subspace $\dot\sblv^{1}(D_{\mathbf{z}};\Hom^{0,1}(TD_{\mathbf{z}},u^{\ast}(T\C^{n})))$ of complex anti-linear maps $A$ with vanishing trace $A|_{\pa D_{\mathbf{z}}}$ in $\sblv^{\frac12}(\pa D_{\mathbf{z}})$.  These spaces fit together to give a bundle $\YY_{p,q}\to\XX_{p,q}$, and  Equation \eqref{Eqn:CR-1} can be viewed as giving a Fredholm section of that bundle.

We write $\FF_{p,q}^{r}$ for the space of Floer disks with $p$ positive and $q$ negative punctures satisfying \eqref{Eqn:CR-1} with parameter $r$. Let $\FF_{p,q}=\bigcup_{r\in[0,\infty)}\FF_{p,q}^{r}$.  Finally, let $\MM_{p,q}$ be the moduli space of holomorphic disks (solutions to the unperturbed Cauchy-Riemann equation) with $p$ positive and $q$ negative punctures, modulo automorphism.

\begin{Lemma} \label{Lem:formaldim}
The formal dimensions of the solution spaces are the following:
\begin{align}
\dim(\FF_{p,q}^{r})&=2p+(1-q)n,\\
\dim(\FF_{p,q})&=2p+(1-q)n+1,\\
\dim(\MM_{p,q})&=(n-3)+2p-nq. 
\end{align}
\end{Lemma}

\begin{proof}
Since $\Sigma$ is a homotopy sphere the Maslov index of $f$ vanishes. The index of the linearized problem corresponding to a holomorphic disk with a positive puncture at $a$ of grading $|a|=n$ was computed in \cite{EES1} to be $n+1$. The formulae above then follow from repeated application of the additivity of the index under gluing and the fact that the dimension of the space of conformal structures on a disk with $p+q$ punctures equals $p+q-3$. 
\end{proof}

The perturbation theory for obtaining transversality for the spaces $\FF_{p,q}$ is straightforward, see Section \ref{app:CRFredholm}, and gives the following result. 

\begin{Lemma}\label{Lem:actualdimensions}
For generic $J \in \JJ_\Sigma$ and generic displacing Hamiltonian,
$\FF_{p,q}$ is a smooth transversely cut out manifold of dimension $2p+(1-q)n+1$. 
\end{Lemma}


Below we will only use Floer disks with negative punctures; to simplify notation we write $\FF_{q}=\FF_{0,q}$. Lemma \ref{Lem:actualdimensions}  shows that for generic data $\FF_{q}=\varnothing$ for $q>1$, $\FF_{1}$ is $1$-dimensional, and $\FF_{0}$ is $(n+1)$-dimensional. 

Corollary \ref{Cor:Constants} implies that  the space $\FF_{0}^{0}$ of constant maps is transversely cut out and is the boundary of $\FF_{0}$.

\subsection{Breaking} \label{Sec:Breaking}

Let $\MM = \MM_{1,0}$ denote the space of holomorphic disks with exactly one positive puncture.  Let $G_1 \subset PSL_2(\bR)$ denote the subgroup of maps fixing $1\in \partial D$.
The space $\MM$ is a quotient of a space of parametrized holomorphic disks by the reparametrization action of $G_1$, hence is not directly a submanifold of a Banach space.

\begin{Lemma} \label{Lem:holdisks1puncture}
For generic $J \in \scrJ_\Sigma$ and displacing Hamiltonian, the moduli space $\MM$ is a compact $C^1$-smooth manifold of dimension $n-1$.
\end{Lemma}

\begin{proof}
The existence of a smooth structure on $\MM$ follows from a gauge-fixing procedure which we give in Section \ref{ssec:1+gauge}. It realizes $\MM$ as a submanifold of a Banach manifold, cf.~Equation \eqref{Eqn:MMsmooth}. Its dimension is then given by the virtual dimension of Lemma \ref{Lem:formaldim}. 
 By Stokes' theorem, disks in $\MM$ have minimal area amongst non-constant holomorphic disks, whence Gromov compactness \cite{Oh:bubble} implies that $\MM$ is compact. 
 \end{proof}

\begin{Lemma}\label{lemma:weakmoduli}
If $J \in \JJ_\Sigma$ and $H\colon\C^{n}\times[0,1]\to\R$ are generic, then the space $\FF_{1}$ is a compact $C^{1}$-manifold of dimension $1$, and $\FF_{0}$ is a non-compact $C^{1}$-manifold of dimension $n+1$ with Gromov-Floer boundary diffeomorphic to the product $\FF_{1}\times\MM$.
\end{Lemma}

\begin{proof}
By Gromov compactness, any sequence in $\FF_{q}$ has a subsequence converging to a broken disk in $\FF_{q+j}$ with $j$ holomorphic disks from $\MM$ attached at negative punctures. By Lemma \ref{Lem:actualdimensions}, $\FF_{q}=\varnothing$ if $q>1$. The first statement follows. For the second statement, the argument just given shows that elements in $\FF_{1}\times \MM$ are the only possible broken configurations in the Gromov-Floer boundary. A standard gluing argument, given in detail in Appendix \ref{Appendix} (and concluded in Section \ref{Ssec:NhoodBdary}) shows that all elements in the product appear as limits of sequences in $\FF_{0}$. 
\end{proof}

Lemma \ref{lemma:weakmoduli} refers, via the gauge-fixing procedure alluded to in Lemma \ref{Lem:holdisks1puncture} and discussed in detail in Section \ref{ssec:1+gauge}, only to the smoothness of the Gromov-Floer boundary of the compactification of $\FF_0$, and not to the smoothness of the compactified moduli space itself.  The following theorem asserts the stronger conclusion. 

\begin{Theorem}\label{Thm:C1boundary}
For generic $J \in \JJ_\Sigma$ and Hamiltonian $H\colon\C^n\times[0,1]\to\R$, the compactified moduli space of Floer disks
\[
\overline{\FF}_{0}=\FF_{0}\cup (\FF_{1}\times\MM)
\]
admits the structure of a $C^{1}$-manifold with boundary, $\pa\overline{\FF}_{0}=\Sigma\, \sqcup\, (\FF_{1}\times\MM)$.
\end{Theorem}

By a classical result in differential topology, see e.g.~\cite{Hirsch}, any $C^{1}$-atlas on a finite dimensional manifold contains a uniquely determined compatible $C^{\infty}$-atlas.  

A detailed and essentially self-contained proof of Theorem \ref{Thm:C1boundary}, which is somewhat simpler than its counterpart in \cite{EkholmSmith}, occupies Appendix \ref{Appendix}.  In broad outline, cf.~Section \ref{Ssec:Overview}, a Newton iteration and rescaling argument gives a smooth embedding $\Psi$ of $\FF_{1}\times\MM\times[\rho_0,\infty)$ into $\FF_{0}$ that we show covers a neighborhood $U_M= \{u\in\FF_{0}\colon \sup_{z\in D}|du(z)|>M\}
$ of the Gromov-Floer boundary.  The embedding $\Psi$ associates to a triple $(u^-, u^+, \rho)$ a Floer holomorphic disk on a domain $\Delta_{\rho} = D^- \#_{\rho} D^+$ which contains a long neck $[-\rho,\rho]\times[0,1]$, cf.~Section \ref{ssec:glueddomains}.  The desired smooth manifold with boundary $\overline{\FF}_0$ is obtained from the two pieces $\FF_0-\overline{U_M}$, for $M>0$ large, and $\FF_1\times\MM\times[\rho_0,\rho_1]$ for $\rho_0, \rho_1 \gg 0$ sufficiently large; in particular, there is a $C^1$-smooth model for $\overline{\FF}_0$ which does not contain  the broken solutions themselves.

\subsection{Tangent bundles and index bundles}\label{Sec:tangentbdles}

Recall the configuration spaces $\XX_{p,q}$ for (Floer) holomorphic disks with boundary on $\Sigma$, with $p$ positive and $q$ negative boundary punctures, modelled on Sobolev spaces with two derivatives in $L^{2}$.  It is standard that the $KO$-theory class of the tangent bundle of the solution space $\FF_{p,q}$ and $\MM_{p,q}$ of the perturbed and unperturbed $\bar\pa$-equation, respectively, is given by the linearization of the $\cdbar$-operator, which in turn is globally restricted from the ambient configuration space  
\[
\left[ T\XX_{p,q} \xrightarrow{D(\bar\pa)}\YY_{p,q} \right],
\]
where $\YY_{p,q}$ is the Sobolev space of complex anti-linear maps. More precisely, for a Lagrangian immersion $L \rightarrow \bC^n$, any map $D_{p,q} \rightarrow \bC^n$ with boundary on $L$ defines a $\bar\pa$-operator in a trivial bundle over the disk. This is a formally self-adjoint Fredholm operator, and hence there is a polarized Hilbert bundle (polarized by the positive and negative spectrum) over the configuration space of all smooth disk maps.  Any such bundle is classified by a map 
\begin{equation} \label{Eq:Classifyindex}
K_{p,q}: \textrm{Maps}((D_{p,q}, \partial D_{p,q}), (\bC^n, L)) \longrightarrow BGL_{res} \simeq U/O
\end{equation}
into the restricted Grassmannian of Hilbert space, which by \cite{PressleySegal} is homotopy equivalent to $U/O$.  (The induced map on constants is exactly the stable Gauss map of the immersion.)  The map \eqref{Eq:Classifyindex} tautologically classifies the index bundle of the $\bar\pa$-operator over configuration space. 

The embedding $\Psi$ mentioned after Theorem \ref{Thm:C1boundary} is the composition of a pre-gluing map $PG: \FF_1\times\MM\times[\rho_0,\infty) \rightarrow \XX_{0,0}$ with a Newton iteration and a rescaling.  The underlying pre-gluing map, which splices disks with positive and negative punctures $D^{\pm}$ to obtain the domain $\Delta_{\rho} = D^- \#_{\rho} D^+$ containing a long neck, can be extended to arbitrarily large compact subsets of the configuration spaces themselves, at the cost of increasing $\rho$.  Slightly abusively, we denote by 
\begin{equation} \label{Eqn:ConfigSpacePregluing}
\Pre: \XX_{0,1} \times \XX_{1,0} \longrightarrow \XX_{0,0}
\end{equation}
the resulting map, which is well-defined up to homotopy on any compact subset of the domain.

\begin{Lemma} \label{Lem:GlueIndex}
The pregluing map of \eqref{Eqn:ConfigSpacePregluing} is compatible with the classifying maps \eqref{Eq:Classifyindex} of the index bundles, i.e.  $K_{0,0} \circ PG \simeq K_{0,1}\times K_{1,0}$ are homotopic as maps  (on compact subsets) $\XX_{0,1} \times \XX_{1,0}\to U/O$.
\end{Lemma}

\begin{proof} 
This follows from the linear gluing argument of Lemma \ref{Lem:partialglu1}.  Indeed, over compact subsets of configuration space, we can add 
 trivial bundles $\R^{m_1}$ and $\R^{m}$ to the domains of $\bar\pa$-operators in $\XX_{0,1}$ and $\XX_{1,0}$, stabilising them so they become everywhere surjective.  At a surjective operator, the fiber of the index bundle is given by the kernel. For $\rho$ large enough, Lemma \ref{Lem:partialglu1} provides an isomorphism
\begin{equation} \label{Eqn:linearglue}
G_{\rho}\colon\ker(D\bar\pa\oplus\R^{m_1})_{u^-}\oplus\ker(D\bar\pa\oplus\R^{m})_{u^+}
\to
\ker(D\bar\pa\oplus\R^{m_1}\oplus\R^{m})_{u^-\#_{\rho} u^+}
\end{equation}
which takes vector fields in the kernels of the operators $u^{\pm}$  to cut-off versions supported in the respective halves of the pre-glued domain.  The result follows.
\end{proof}

Lemma \ref{Lem:GlueIndex} and Theorem \ref{Thm:C1boundary} imply that the tangent bundle of the compactified manifold $\overline{\FF}_0$ is given, in $KO$-theory, by a class restricted from the configuration space $\XX_{0,0}$.


\section{Constructing a parallelizable bounding manifold}\label{Sec:parallel}
In this section we show that if the stable Gauss map $Gf$ of $f:\Sigma \rightarrow \bC^n$ is homotopic to that of the Whitney immersion, then $\Sigma \in bP_{n+1}$, i.e.~$\Sigma$ bounds a parallelizable $(n+1)$-manifold. This excludes numerous diffeomorphism types by results of \cite{KM}. We obtain the bounding manifold for $\Sigma$ by attaching a cap to the outer boundary $\FF_{0}^{\bd}=\pa\overline{\FF}_{0}-\Sigma\approx\FF_{1}\times\MM$. We do this in the simplest possible way: let
\[
\scrB = \overline{\FF}_0 \cup_{\FF^{\bd}_0} (\DD\times\MM)
\]
where $\DD = \sqcup D^2$ is a finite collection of closed disks abstractly bounding the  1-manifold $\FF_1$.

\begin{Proposition}
The space $\scrB$ admits the structure of a compact $C^1$-smooth manifold, with boundary canonically diffeomorphic to the homotopy sphere $\Sigma$.
\end{Proposition}

\begin{proof}
This is an immediate consequence of Theorem \ref{Thm:C1boundary}.
\end{proof}

We claim that, under the homotopy assumption on the Gauss map, the manifold  $\scrB$ constructed above has trivial tangent bundle.  Since, for manifolds with boundary, triviality and stable triviality of the tangent bundle are equivalent,  it suffices to show that $T\scrB$ is trivial in the reduced $KO$-theory $\widetilde{KO}(\scrB)$.

\subsection{The Whitney immersion} \label{Sec:whitney}
To analyze $T\scrB$, following \cite{EkholmSmith} and Section \ref{Sec:tangentbdles}, we use index theory and aim to trivialize the relevant index bundles over configuration spaces of smooth maps of punctured disks. For these to be sufficiently explicit as to be amenable to study, we reduce to a model situation.

\begin{Definition}\label{def:Whitney}
Let $W\colon S^{n}\to\C^{n}$ denote the Whitney immersion,  
\[
W\colon \{(x,y) \in \bR^n \times \bR \, | \, |x|^2 + y^2 = 1 \} = S^n \to \bC^n \quad W(x,y)=(1+iy)x.
\]
\end{Definition}

This is a Lagrangian immersion of a sphere with a unique double point, for which the corresponding Reeb chord has Maslov index $n$. Consider now a homotopy $n$-sphere $\Sigma$ and a Lagrangian immersion $f\colon\Sigma\to \C^{n}$ with one double point $a$ of index $n$. Fix once and for all a stable framing of the tangent bundle of $\Sigma$, which exists by \cite{KM}, and fix a Hamiltonian isotopy of $\C^{n}$ that makes $f$ agree with the Whitney immersion in a neighborhood of the double point $a$.  

In this situation, any choice of stable framing $\eta$ of the tangent bundle $TS^n$ of the standard sphere induces a \emph{difference element} $\delta(\eta) \in \pi_n U$, obtained as follows.   Since $f$ and $W$ have the same Maslov index, the maps may be homotoped (relative to a neighborhood of the double point) so as to co-incide in a neighborhood of a path connecting the double point preimages.  The complement of such a path is an $n$-disk in either $S^n$ or $\Sigma$. Since the tangent bundles are framed, the Gauss maps $df$ and $dW$ are naturally valued in $U(n)$, and co-incide on the boundaries of these disks. They therefore patch to define a map $\delta: S^n \rightarrow U(n)$, well-defined up to homotopy. The difference element $\delta(\eta)$ is the corresponding stable map $S^n \rightarrow U$.

\begin{Lemma}\label{lemma:deform1}
There exists a $\mathrm{PL}$-isotopy $\Phi_t\colon\Sigma\to\C^{n}$ fixed around $f(a)$ such that $\Phi_0=\id$ and $\Phi_1\circ f$ is a parameterization of the Whitney immersion. Furthermore, if the stable Gauss map $Gf$ is homotopic to the stable Gauss map $GW$ of the Whitney immersion, there is a stable framing $\eta$ of $TS^n$ so that the induced difference element $\delta(\eta) = 0$. 
\end{Lemma}
  
\begin{proof}
The first statement follows from a PL-isotopy extension theorem due to Hudson and Zeeman \cite{HudsonZeeman}, compare to \cite[Lemma 3.13]{EkholmSmith}.  For the second statement, note that changing the stable framing $\eta$ on the standard sphere corresponds to precomposing the map $\delta$ with a map into $O$. This is sufficient to kill  the stable obstruction $\delta(\eta)$ precisely when the stable Gauss maps into $U/O$ are homotopic.
\end{proof}

\begin{Corollary}\label{corollary:deform1}
If the stable Gauss maps of $f$ and $W$ are homotopic, there is a one-parameter family of $C^0$-homeomorphisms $\phi_t\colon \bC^n \rightarrow \bC^n$ and a 1-parameter family of continuous maps $\psi_t: S^n \rightarrow U$ with the properties that
\begin{enumerate}
\item $\phi_0$ is the identity;
\item $\phi_1(f(\Sigma)) = W(S^n)$;
\item $\psi_0 = df$ and $\psi_1 = dW \circ \phi_1$.
\end{enumerate}
\end{Corollary}

We combine the previous Corollary with the discussion of Section \ref{Sec:tangentbdles}.  By conjugating with the formal isotopy of Lagrangian planes produced by the final part of Corollary \ref{corollary:deform1}, the maps \eqref{Eq:Classifyindex} associated to the immersion $f$ of $\Sigma$ and the Whitney immersion are homotopic.  This proves:

\begin{Corollary} \label{Cor:stableindextrivial} The index bundle on the configuration space $\XX_{p,q}$  is stably trivial if and only if the index bundle on the corresponding configuration space for the Whitney immersion is stably trivial.\qed
\end{Corollary}

\subsection{Index bundles and configuration spaces}

We restrict attention to the configuration spaces relevant to our study and write $\XX_{0}=\XX_{0,0}$, $\XX_{1}=\XX_{0,1}$, and $\XX=\XX_{1,0}$ for the configuration spaces that contain $\FF_{0}$, $\FF_{1}$, and $\MM$, respectively.  

For a compact connected cell complex $K$ and $x,y\in K$ let $\Omega(K;x,y)$ denote the space of paths in $K$ from $x$ to $y$, which is homotopy equivalent to the based loop space $\Omega(K)=\Omega(K;x,x)$ of $K$. Let $\LL(K)$ denote the free loop space of $K$. Recall that $\{a^{+},a^{-}\} \subset \Sigma$ are the preimages of the double point $a\in\C^{n}$ of the immersion $f$.

\begin{Lemma}\label{lemma:hconfig}
There are homotopy equivalences
\[
\XX_{1}\simeq S^{1}\times\Omega(\Sigma;a^{-},a^+); \quad
\XX\simeq \Omega(\Sigma;a^{+},a^{-});\quad
\XX_{0}\simeq \LL(\Sigma),
\]
In particular, the homotopy types of the configuration spaces are independent of the smooth structure on the homotopy sphere $\Sigma$.  
\end{Lemma}

\begin{proof}
In all cases, the homotopy fiber of the restriction to the boundary is contractible via linear homotopies. Note that the $S^{1}$-factor encodes the location of the boundary puncture of maps in $\XX_{1}$. The final statement follows since the homotopy type of the based and free loop spaces is a homotopy invariant of compact cell complexes.
\end{proof}

 Let $Q \rightarrow U\Sigma$ denote the vertical tangent bundle to the unit sphere bundle $U\Sigma \subset T\Sigma$, and $\widehat{Q}$ the fiberwise Thom space of $Q$, so the fiber $\widehat{Q}_{\sigma}$ of $\widehat{Q}$ over $\sigma\in\Sigma$ is the Thom space of the tangent bundle $T(U_\sigma\Sigma)$ of the fiber of the unit sphere bundle at $\sigma$.  Note that the Thom space contains a distinguished point $\infty$.  Let
\[
\Sigma\vee \widehat{Q}_{\sigma} =\Sigma \cup \widehat{Q}_{\sigma}/\sigma\sim\infty.
\]

\begin{Lemma}\label{Lem:Skeleta}
There are embeddings
\[
\iota\colon S^{n-1} \to \Omega(\Sigma;a^{\pm},a^{\mp}) \qquad \text{and} \qquad \kappa\colon \Sigma\vee\widehat{Q}_{\sigma}\ \to \LL(\Sigma)
\]
which give a $2n-3$ respectively $3n-4$ skeleton for the given spaces. In other words $\iota$ is $(2n-3)$-connected and $\kappa$ is $(3n-4)$-connected.  
\end{Lemma}

\begin{proof}
The skeleta arise from the usual Morse-Bott theory of the energy functional for the round metric on the standard sphere $S^n$.  For the path space $\Omega(\Sigma;a^{\pm},a^{\mp})$, take $a^{+}$ and $a^{-}$ to be antipodal points. Then the energy functional is Bott-degenerate with Bott-manifolds diffeomorphic to $S^{n-1}$ of indices $2k(n-1)$, $k=0,1,2,\dots$. It follows that the Bott-manifold of index $0$ gives a $(2n-3)$-skeleton. 

For the free loop space, consider the fibration $\LL(\Sigma)\to\Sigma$ with fiber the based loop space $\Omega(\Sigma)$. Again the energy functional is Bott-degenerate with critical manifold $\Sigma$ of index $0$ corresponding to constant loops, and critical manifolds $U\Sigma$ of indices $2k(n-1)$, $k=1,2,3,\dots$.
The closure of the unstable manifold of the index $2(n-1)$ critical manifold $U\Sigma$ gives an embedding of $\widehat{Q}$, see \cite[Lemma 7.1]{EkholmSmith}, which then gives a $(3n-3)$-skeleton. Removing the fiber over one point we get  a $(3n-4)$-skeleton homotopy equivalent to $\Sigma\vee\widehat{Q}_{\sigma}$.
\end{proof}

Combining Lemmas \ref{lemma:hconfig} and \ref{Lem:Skeleta} we get a corresponding $(2n-3)$-skeleton $S^{n-1}\approx K^{2n-3}\subset \XX$, a $(2n-2)$-skeleton $S^{1}\times S^{n-1}\approx K^{2n-2}_{1}\subset \XX_1$, and a $(3n-4)$-skeleton $\Sigma\vee\widehat{Q}_{\sigma}\approx K^{3n-4}_{0}\subset \XX_{0}$, consisting of disk maps with boundary paths in $\iota(S^{n-1})$, with boundary paths in $\iota(S^{n-1})$ and base point according to the $S^{1}$-factor, and with boundary loops in $\kappa(\widehat{Q})$, respectively. 

We write $I_{0}$, $I_{1}$, and $I$ to denote the index bundles over $\XX_{0}$, $\XX_{1}$, and $\XX$ respectively.

\begin{Lemma}\label{lemma:strivskel}
The index bundles $I_0$, $I$, and $I_1$ are stably trivial over the skeleta $K_{0}^{3n-4}$, $K^{2n-3}$, and $K^{2n-2}_1$, respectively.
\end{Lemma}

\begin{proof}
By Corollary \ref{Cor:stableindextrivial}, it is sufficient to find stable trivializations of the index bundles for the Whitney immersion, and we can moreover assume that we work with the undeformed Cauchy-Riemann equation and the standard $\bar\pa$-operator.   We use a model of the Whitney sphere that is almost flat, i.e.~for some small $\epsilon > 0$, we scale $\R^{n}$ and $i\R^{n}$ in the model of Definition \ref{def:Whitney} by $\epsilon^{-1}$ and $\epsilon$, respectively.

Consider first $I\to K^{2n-3}$. The Lagrangian boundary condition corresponding to a geodesic  connecting $a^{+}$ to $a^{-}$ splits into a 1-dimensional problem, corresponding to the tangent line of the curve in Figure \ref{Fig:Bdry}, of index $2$ with 2-dimensional kernel, and an $(n-1)$-dimensional problem with almost constant boundary conditions and hence with $(n-1)$-dimensional kernel. It follows that the index bundle is stably equivalent to $T S^{n-1}$, in particular stably trivial as claimed.     
\begin{center}
\begin{figure}[ht]
\includegraphics[scale=0.5]{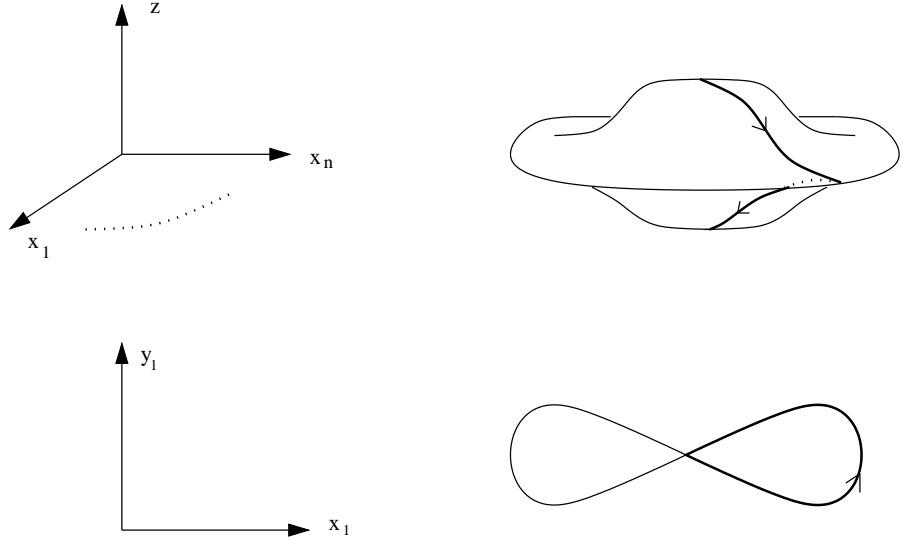}
\caption{The Lagrangian boundary condition along a disk with a positive puncture\label{Fig:Bdry}}
\end{figure}
\end{center}

For $I_{1}\to K_{1}^{2n-2}$ the boundary condition  splits in the same way, except now the $1$-dimensional problem corresponds to the curve in Figure \ref{Fig:Bdry} oriented in the opposite direction, since the puncture is negative. The index for that problem now equals $-1$, whilst the $\bar\pa$-operator on the almost constant $(n-1)$-dimensional orthogonal boundary condition  is an isomorphism. Thus, the index bundle is trivial.

Finally, consider $I_{0}\to K_{0}^{3n-4}$.  Restricted to constant loops in $S^n$, the index bundle is isomorphic to the tangent bundle, hence is stably trivial. Consider now $\widehat{Q}_{\sigma}$.  Taking $\sigma=a^{+}$ we find that the geodesic loops are unions of the paths considered above. Using sufficiently stretched disks of the form $\Delta_{\rho} = D^-\#_{\rho} D^+$, cf. the discussion after Theorem \ref{Thm:C1boundary} and Section \ref{ssec:glueddomains},  the boundary condition over $U_{\sigma}S^{n}$ can be considered as a sum of the two problems studied above. Stabilizing the operator defining $I_{1}$ once, the corresponding operator becomes an isomorphism. By considering an explicit model of the Morse flow, compare to the proof of \cite[Lemma 7.3]{EkholmSmith},  one sees that 
the restriction of $I_0$ to $\widehat{Q}_{\sigma}=MT(U_{\sigma}S^{n})$ is stably isomorphic to the pull-back of the tangent bundle $T(U_{\sigma}S^{n})$, hence is  stably trivial. The trivialization can clearly be glued to that  over $S^n$ at $\sigma$, which completes the proof.
\end{proof}

\subsection{Coherent stable trivializations}
Consider stable trivializations $Z_{0}$, $Z_{1}$, and $Z$ of $I_{0}$, $I_{1}$, and $I$, respectively. Recall from Section \ref{Sec:tangentbdles} that for any compact subsets $B_1\subset\XX_{1}$ and $B\subset\XX$ there is a pre-gluing map 
\[
\Pre_{\rho}\colon B_{1}\times B\to\XX_{0}
\]
for all sufficiently large $\rho$. Suppose, as in Lemma \ref{Lem:GlueIndex}, we stabilize the linearized $\bar\pa$-operators defining $I_{1}$ and $I$ over $B_{1}$ and $B$ by adding trivial bundles $\R^{m_1}$ and $\R^{m}$ to make the operators everywhere surjective.  Recall the isomorphism $G_{\rho}$ 
\begin{equation} 
G_{\rho}\colon\ker(D\bar\pa\oplus\R^{m_1})_{u_1}\oplus\ker(D\bar\pa\oplus\R^{m})_{u_2}
\to
\ker(D\bar\pa\oplus\R^{m_1}\oplus\R^{m})_{u_1\#_{\rho} u_2}
\end{equation}
 of Equation \eqref{Eqn:linearglue}.  We say that stable trivializations $(Z_1,Z,Z_0)$ are $(d_1,d)$-coherent if for any compact CW-complexes $B_{1}\subset \XX_{1}$ and $B\subset \XX$ of dimensions $d_1$ and $d$, respectively,  the stable trivializations $\Pre^{\ast}_{\rho}Z_{0}$ and $G_{\rho}(Z_1\oplus Z)$ of the pull-back of the index bundle $\Pre^{\ast}I_0$ are homotopic.  Coherent trivializations generalize the more familiar notion of coherent orientations.

\begin{Lemma}\label{lemma:coherent}
For any stable trivialization $Z_0$ of $I_0$ and $Z_1$ of $I_1$, with $Z_1$ trivial over the $1$-skeleton (i.e.~over the $S^1$-factor in $K_{1}^{2n-2}\approx S^{1}\times S^{n-1}$), there is a stable trivialization 
$Z$ of $I$ such that $(Z_1,Z,Z_0)$ is $(1,n-1)$ coherent. 
\end{Lemma}

\begin{proof}
After homotopy, we may assume that any closed $1$-complex $B_1$ mapping to $\XX_{1}$ maps into the $1$-skeleton $S^{1}\times v$, for some $v\in S^{n-1}$, and that any compact $(n-1)$-complex mapping to $\XX$ maps into $K^{2n-1}\approx S^{n-1}$. Recall from Lemma \ref{lemma:strivskel} that the trivialization $Z_0$ of $I_0$ over $K_{0}^{3n-4}$ is constructed using almost broken disks. By definition, in such a trivialization $Z_0=G(Z_1\oplus Z)$, where we stabilize the operator defining $I_{1}$ once to make it an isomorphism over $S^{n-1}$, and where $Z$ is some given trivialization of $I$. \end{proof}

\subsection{Proof of Theorem \ref{thm:main}}
Lemma \ref{lemma:hmtpysphere} shows that $K$ is a homotopy sphere. The regular homotopy assumption together with Lemma \ref{lemma:strivskel} implies that the index bundles $I_1$, $I_0$, and $I$ are stably trivial, and using Lemma \ref{lemma:coherent} we find a coherent triple $(Z_1,Z,Z_0)$ of stable trivializations where $Z_1$ is trivial over the $1$-skeleton. Note that the tangent bundle of the $1$-manifold $\FF_{1}$ is canonically isomorphic to $I_1$. Hence, our assumption on the stable trivialization $Z_1$ over the $1$-skeleton implies that the induced stable trivialization of $T\FF_{1}$ extends to a trivialization $\tilde Z_1$ of $T\DD$. We thus get a stable trivialization $\tilde Z_1\oplus Z$ of $T(\DD\times\MM)$, which near $\pa^{\bd}\FF_0$ agrees with $G(Z_1\oplus Z)$ and hence extends as $Z$ to a stable trivialization of $T\FF_0$. We conclude that $T\BB$ is stably trivial.\qed

\subsection{Signature}\label{Sec:signature}

Any (not necessarily closed) $4k$-manifold $X$ has a signature $\sigma(X)$ which is the number of positive minus the number of negative eigenvalues of the intersection form on $H_{2k}(X;\bQ)$.  According to \cite{KM}, the exotic nature of a $(4k-1)$-dimensional homotopy sphere is reflected in the possible signatures of parallelizable manifolds which it bounds.  

Now suppose, as in previous sections, that  $H\colon \bC^n \times [0,1] \rightarrow \bR$ is a compactly supported Hamiltonian function whose time 1 flow displaces a Lagrangian immersion $f\colon\Sigma\to\C^{n}$ with one double point of index $n$. Let $\BB_H$ denote the bounding manifold, with $\pa \BB_H =\Sigma$, obtained from capping the space of Floer solutions as usual, so
$\BB_H =\overline{\FF}_{0}\cup_{\FF_1\times\MM}(\DD\times\MM).$

\begin{Proposition}
$\sigma(\BB_H)$ is independent of $H$, hence defines an invariant $\sigma(f)$ of the original Lagrangian immersion $f: \Sigma \rightarrow \bC^n$.
 If $W: S^n \rightarrow \bC^n$ is the Whitney immersion, then $\sigma(W) = 0$.
\end{Proposition}

\begin{proof}[Sketch] Both statements are applications of Novikov additivity; we just sketch the main ideas. The first result is obtained by considering a family of displacing Hamiltonians $\{H_s\}$ interpolating two given such $H_0, H_1$, and a corresponding family of Floer-equations \eqref{Eqn:CR-1}$_R$ parametrized by $R \in \frak{h}$, a half-plane, see Figure \ref{fig:HamiltHalfplane}.  
\begin{figure}[ht]
\centering
\includegraphics[width=0.6\linewidth, angle=0]{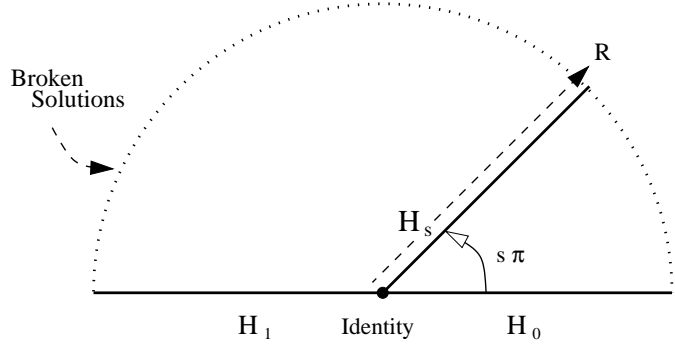} 
\caption{Cauchy-Riemann problems parametrized by a half-plane}
\label{fig:HamiltHalfplane}
\end{figure}  
The Gromov-Floer compactified space of Floer disks with no punctures over $\frak{h}$ admits a cap of the shape $S \times \MM$ for a solid handlebody $S$, and shows that $\BB_{H_0} \cup_{\Sigma} \BB_{H_1}$ is a global boundary, hence has trivial signature. (For this argument, $C^1$-structures and parallelizability play no role.)  The second result follows from a similar construction, using the fact that the Whitney immersion of the $n$-sphere extends to an exact Lagrangian immersion of an $(n+1)$-ball.  We leave details to the reader. 
\end{proof}

There is a natural evaluation map $ev_1: \FF_0 \rightarrow \bC^n$ which extends a given immersion $f$,  but in general there is no reason to expect $ev_1$ to be better-behaved than a generic map.  The regular homotopy classes of Lagrangian immersions in Euclidean space of a homotopy $(8k-1)$-sphere are classified by $\pi_{8k-1}(U) \cong \bZ$.  The set of elements realized by immersions with one double point of Maslov grading $n$ of some given homotopy sphere $\Sigma$ may be related, via the signature of the bounding manifold $\BB$, to the singularities of the evaluation map $ev_1$ on high-dimensional moduli spaces of Floer disks.  It would be interesting to understand if those singularities can be interpreted in terms of the local analysis of the underlying Floer disks themselves. 


\appendix
\section{Smooth structures on compactified moduli spaces} \label{Appendix}
In this appendix we present  the details of the arguments needed for the gluing results leading to Theorem \ref{Thm:C1boundary} in Section \ref{Sec:modulispaces}. Much of this material appears in similar form elsewhere, but is collected here in a streamlined presentation that fits our needs. 

\subsection{Functional analytic set up}\label{Ssec:functional-analytic}
In this section we describe in greater detail the functional analytical set up that we use to derive properties of solution spaces of (perturbed) Cauchy-Riemann equations. We first recall the set up from Section \ref{Sec:modulispaces}: $f\colon \Sigma\to \C^{n}$ is a Lagrangian immersion such that $f(\Sigma)$ is real analytic and such that the double point of $f$ lies at $0\in \bC^n$, the branches of $f(\Sigma)$ near the double point are flat and co-incide with $\bR^n \cup i\bR^n$ in a sufficiently small ball around the origin. Let $\tilde f=(f,h)\colon \Sigma\to \C^{n}\times\R$ denote the Legendrian lift of $f$ into $\C^{n}\times\R$ with lower Reeb chord endpoint at level $0$ in $\R$. Also, $\mathcal{J}_{\Sigma}$ is the space of smooth almost complex structures that agree with the standard complex structure $J_{0}$ in a small fixed neighborhood of $f(\Sigma)$. In this section, we write $\sblv^{k}$ for the Sobolev space of maps, or sections in a bundle, with $k$ derivatives in $L^{2}$.  

We first introduce the configuration spaces and associated spaces that we will use for various domains, and then turn to their structures as Banach manifolds and Banach bundles.

\subsubsection{Configuration spaces for the closed disk}\label{ssec:closeddisk}
Let $D$ denote the closed unit disk in the complex plane. Consider the space $\sblv^{2}(D;\C^{n})$ of maps $u\colon D\to\C^{n}$ with two derivatives in $L^{2}$ with the standard norm $\|u\|_{2}$. Note in particular that elements in $\sblv^{2}(D;\C^{n})$ are representable by continuous maps. For $J\in \mathcal{J}_{\Sigma}$ and $u\in\sblv^{2}(D;\C^{n})$ let $\sblv^{1}(\Hom^{0,1}(TD,u^{\ast}(T\C^{n})))$ denote the Sobolev space of $(J,i)$-complex anti-linear maps $TD\to u^{\ast}(T\C^{n})$ with one derivative in $L^{2}$.
Then for $u\in \sblv^{2}(D;\C^{n})$, $du^{0,1}\in\sblv^{1}(\Hom^{0,1}(TD,u^{\ast}(T\C^{n})))$ and consequently the trace (the ``restriction'' to the boundary) $du^{0,1}|_{\pa D}$ lies in the Sobolev space $\sblv^{\frac{3}{2}}$ on $\pa D$. 

Define $\XX\subset \sblv^{2}(D;\C^{n})\times \sblv^{\frac32}(\pa D;\R)$ to be the subset of pairs $(u,h)$, where $u\colon D\to \C^{n}$ and $h\colon\pa D\to\R$ satisfies the following conditions:
\begin{enumerate}
\item $(u|_{\pa D},h)\colon \pa D\to \C^{n}\times\R$ takes values in the Legendrian lift $\tilde f(\Sigma)$ of $f\colon \Sigma\to\C^{n}$; 
\item $du^{0,1}|_{\pa D}=0$;
\end{enumerate}
The space $\XX$ will serve as our configuration space. 

We will also use a bundle $\YY$ over $\XX$. To define it, let $(u,h) \in\XX$ and consider the Sobolev space $\sblv^{1}(\Hom^{0,1}(TD,u^{\ast}(T\C^{n})))$ of complex anti-linear maps. The fiber $\YY_u$ is then the closed subspace 
\[
\dot\sblv^{1}(\Hom^{0,1}(TD,u^{\ast}(T\C^{n})))\subset
\sblv^{1}(\Hom^{0,1}(TD,u^{\ast}(T\C^{n})))
\]
of complex anti-linear maps $A$ with vanishing trace $A|_{\pa D}$ in $\sblv^{\frac12}(\pa D)$. 

\subsubsection{Configuration spaces for disks with punctures and marked points}\label{ssec:punctureddisks}
We next consider disks with punctures. Let $\mathbf{z}=(z_1,\dots,z_m)$ be a collection of $m$ distinct points in $\pa D$ cyclically ordered according to the orientation of $\pa D$. Consider the function 
$\hat\delta(z_j)\in\R_{>0}$, which equals the minimum of the two distances from $z_j$ to its closest neighbors. Fix functions $\delta(z_j)\in\R_{>0}$ that approximate $\hat\delta$ in the sense that $\hat\delta(z_j) >\delta(z_j)>\frac12\hat\delta(z_j)$ and which depend smoothly on $\mathbf{z}$. Using $\delta$ we define a family of metrics on $D-\mathbf{z}$ as follows. For each $z_j$ fix a conformal equivalence $\phi_{j}\colon [-2,\infty)\times[0,1]\to D$ which maps $\infty$ to $z_j$, and takes $(-1,0)$ and $(-1,1)$ to points on $\pa D$ at distance $\frac12\delta(z_j)$ from $z_j$. Define a Riemannian metric on $D-\mathbf{z}$ that agrees with the standard metric on the strip in the regions $\phi_j([0,\infty)\times[0,1])$ around each puncture and which agrees with the standard flat metric on $D$ outside the union of the regions $\phi_j([-1,\infty)\times[0,1])$. We write $D_{\mathbf{z}}$ for the disks with punctures at the points in $\mathbf{z}$ and with this metric.

Consider the Legendrian lift $\tilde f$ of $f$ and note that the $\R$-coordinate of $\tilde f$ is constant in a neighborhood of the double point. By definition the value at the lower branch is $0$. Let $\ell$ denote its value at the upper branch.
Consider next a subdivision $\mathbf{z}=\mathbf{z}_{+}\cup\mathbf{z}_{-}$ and fix a smooth function $h_{0}\colon\pa D_{\mathbf{z}}\to\R$ such that the following hold:
\begin{itemize}
\item If $z_j\in\mathbf{z}_{+}$ then $h_0$ equals $0$ and $\ell$ for points in $\pa D$ at distance $\leq \frac{1}{4}\delta(z_j)$ in the clockwise and counterclockwise direction, respectively, from $z_j$
\item If $z_j\in\mathbf{z}_{-}$ then $h_0$ equals $\ell$ and $0$ for points in $\pa D$ at distance $\leq \frac{1}{4}\delta(z_j)$ in the clockwise and counterclockwise direction, respectively, from $z_j$
\end{itemize}
  
In analogy with Section \ref{ssec:closeddisk}, we define $\XX_{\mathbf{z}}\subset \sblv^{2}(D_{\mathbf{z}};\C^{n})\times \sblv^{\frac32}(\pa D_{\mathbf{z}};\R)$ to be the subset of pairs $(u,h)$, where $u\colon D\to \C^{n}$ and $h\colon\pa D\to\R$, satisfying the following conditions:
\begin{enumerate}
\item $(u|_{\pa D},h+h_0)\colon \pa D\to \C^{n}\times\R$ takes values in the Legendrian lift $\tilde f(\Sigma)$ of $f\colon \Sigma\to\C^{n}$; 
\item $du^{0,1}|_{\pa D}=0$.
\end{enumerate}
The space $\XX_{\mathbf{z}}$ will be our configuration space for punctured disks. We will also use the bundle $\YY_{\mathbf{z}}$ over $\XX_{\mathbf{z}}$ with fiber over $u\in\XX_{\mathbf{z}}$ equal to the closed subspace $\dot\sblv^{1}(\Hom^{0,1}(TD_{\mathbf{z}},u^{\ast}(T\C^{n})))$ of complex anti-linear maps $A$ with vanishing trace $A|_{\pa D_{\mathbf{z}}}$ in $\sblv^{\frac12}(\pa D_{\mathbf{z}})$.
For simpler notation below we will often suppress the $\R$-valued function $h$ from the notation. We point out that this function $h$ is mostly a bookkeeping device which  allows us to preclude from the configuration spaces maps with ``fake'' punctures, where the boundary of the map switches sheets.

When convenient we will often write $D_{p,q}$ for $D_{\mathbf{z}}$ for an unspecified $\mathbf{z}\subset\pa D$ with subdivision $\mathbf{z}=\mathbf{z}_{+}\cup\mathbf{z}_{-}$, where $\mathbf{z}_{+}$ has $p$ elements and $\mathbf{z}_{-}$ has $q$ elements. Analogously, we write $\mathcal{X}_{p,q}$ and $\mathcal{Y}_{p,q}$. Also, to include the case of the closed disk as a special case of our general results below we let $D_{0,0}=D$ and use analogous notation for the associated functional analytic spaces.

Finally, we will also consider the disk with one puncture at $1\in\pa D$ and two marked points at $\pm i\in\pa D$. We endow this disk with a metric as on $D_{\mathbf{z}}$, where $\mathbf{z}$ is the one point set containing only the point $1$.

\subsubsection{A metric for local coordinate charts in configuration spaces}\label{ssec:metric}
We review a construction in \cite[Section 5.2]{EES1}. Consider the restriction $T_\Sigma(T\Sigma)$ of the tangent bundle of $T\Sigma$ to the $0$-section $\Sigma\subset T\Sigma$. Let $J_{\Sigma}\colon T_\Sigma(T\Sigma)\to T_\Sigma(T\Sigma)$ denote the natural complex structure which maps a horizontal vector tangent to $\Sigma\subset T\Sigma$ at $q\in \Sigma$ to the corresponding vector tangent to the fiber $T_q\Sigma\subset T\Sigma$ at $0\in T_q\Sigma$. Using Taylor expansion in the fiber directions, and the special form of $f\colon\Sigma\to\C^{n}$ near the double point, it is straightforward to check that the immersion $f\colon \Sigma\to \C^{n}$ admits an extension $P\colon U\to \C^{n}$, where $U\subset T\Sigma$ is a neighborhood of the $0$-section, such that $P$ is an immersion with $J_0\circ dP = dP\circ J_\Sigma$ along the $0$-section $\Sigma\subset U$.  (Recall that $J_0$ denotes the standard complex structure on $\C^n$.) 

We next construct a metric $\hat g$ on a neighborhood of the $0$-section in $T\Sigma$. Fix a Riemannian metric $g$ on $\Sigma$, which we take to agree with the standard flat metric on a neighborhood of the preimage of the double point $0\in\C^{n}$, where $f(\Sigma)$ agrees with $\R^{n}\cup i\R^{n}\subset \C^{n}$. 
Let $v\in T\Sigma$ with $\pi(v)=q$. Let $X$ be a tangent vector to $T\Sigma$ at $v$. The Levi-Civita connection of  $g$ gives the decomposition $X=X_{H}+X_{V}$, where $X_{V}$ is a vertical tangent vector, tangent to the fiber, and where $X_{H}$ lies in the horizontal subspace at $v$ determined by the connection. Since $X_V$ is a vector in $T_q\Sigma$ with its endpoint at $v\in T_q\Sigma$ we can translate it linearly to the origin $0\in T_q\Sigma$; we also use $X_V$ to denote this translated vector. Write $\pi X\in T_q\Sigma$ for the image of $X$ under the differential of the projection $\pi\colon T\Sigma\to \Sigma$. Let $R$ denote the curvature tensor of $g$ and define the field $\hat g$ of quadratic forms along $T\Sigma$ as follows: 
\begin{equation}
\hat{g}(v)(X,Y)=g(q)(\pi X,\pi Y)+g(q)(X^{V},Y^{V})+g(q)(R(\pi X,v)\pi Y,v),
\end{equation} 
where $v\in T\Sigma$, $\pi(v)=q$, and $X,Y\in T_v(T\Sigma)$.

\begin{Lemma}
There exists $\delta>0$ such that $\hat g$ is a Riemannian metric on $\{v\in T\Sigma\colon g(v,v)< \delta\}$. In this metric the $0$-section $\Sigma$ is totally geodesic and the geodesics of $\hat g$ in $\Sigma$ are exactly those of the original metric $g$. Moreover, if $\gamma$ is a geodesic in $\Sigma$ and $X$ is a vector field in $T(T\Sigma)$ along $\gamma$ then $X$ is a Jacobi field if and only if $J_\Sigma X$ is. 
\end{Lemma}

\begin{proof}
This is \cite[Proposition 5.3]{EES1}.
\end{proof}

Consider the immersion $P\colon U\to \C^{n}$, where $U$ is a neighborhood of the $0$-section in $T\Sigma$, with $J_0\circ dP=dP\circ J_\Sigma$. The push forward under $dP$ of the metric $\hat g$ gives a metric on a neighborhood of $f(\Sigma)$ in $\C^{n}$. (Recall that we use a metric on $\Sigma$ which agrees with the flat metric on the two sheets near the intersection point, which means that the push forwards of the metric on the two sheets agree near the double point.) Extend the metric $\hat g$ to a metric, still denoted $\hat g$, on all of $\C^{n}$ which we take to agree with the standard flat metric on $\C^{n}$ outside a (slightly larger) neighborhood of $f(\Sigma)$. We write $\exp\colon T\C^{n}\to\C^{n}$ for the exponential map in the metric $\hat g$. Since $\Sigma$ is totally geodesic for $\hat g$, $\exp$ takes tangent vectors to $f(\Sigma)$ to points on $f(\Sigma)$.

\subsubsection{Configuration spaces as Banach manifolds}\label{ssec:cfigBanach}
We show that the configuration spaces defined above are Banach manifolds and we provide an atlas with $C^{1}$-smooth transition functions. 

\begin{Lemma}\label{lemma:exponentialmap}
The subset $\XX_{p,q}\subset\sblv^{2}(D_{p,q},\C^{n})\times \sblv^{\frac{3}{2}}(\pa D_{p,q},\R)$ is closed and is a Banach submanifold. The tangent space $T_{u}\XX_{p,q}$ at $u$ is the subspace of vector fields  $v\in\sblv^{2}(D_{p,q},\C^{n})$ such that 
\begin{enumerate}
\item $v|_{\pa D_{p,q}}(z)\in df(T_{\sigma(z)}\Sigma)$, where $\tilde f(\sigma(z))=(u,h+h_0)(z)$ for all $z\in\pa D_{p,q}$ (note that along $\pa D_{p,q}$, $(u,h+h_0)$ gives a map into the embedded submanifold $\tilde f(\Sigma)\subset\C^{n}\times\R$).
\item $(\nabla v)^{0,1}|_{\pa D_{p,q}}=0\in\sblv^{\frac{1}{2}}(\pa D_{p,q},\Hom^{0,1}(TD_{p,q},T\C^{n}))$,
\end{enumerate}
where $\nabla$ denotes the Levi-Civita connection of the metric $\hat g$, see Section \ref{ssec:metric}. Furthermore, the map $\Exp\colon T_{u}\cfig_{p,q}\to\cfig_{p,q}$,
\[
[\Exp(v)](z)=\exp_{u(z)}(v(z)),\quad z\in D_{p,q},
\]
gives local $C^{1}$-coordinates near $u$ when restricted to $v$ in a sufficiently small ball around $0\in T_{u}\cfig_{p,q}$. 
\end{Lemma} 

\begin{proof}
The fact that $\XX_{p,q}$ is a closed subset is a consequence of the fact that the norm in $\sblv^{2}$ controls the sup-norm. The tangent space is obtained from the first order part of the equations defining $\XX_{p,q}$. Finally, it is straightforward to check that the local coordinate map gives a bijection from a small ball in the tangent space at $u$ to an open neighborhood of $u$. The $C^{1}$-property follows from the smoothness of the map $\exp_{x}(y)$. For details, see \cite[Lemma 3.2]{EES} and \cite[Lemmas 5.10 and 5.11]{EES1}.
\end{proof}

With the configuration spaces $\XX_{p,q}$ equipped with $C^{1}$ structures it is straightforward to see that the bundles $\YY_{p,q}$ are locally trivial bundles with $C^{1}$ local trivializations. Explicit such local trivializations using parallel translation in the metric $\hat g$ are given in \cite[Section 3.2]{EES}.

For the transversality results in Section \ref{Sec:transversality} we need to look at parameterized versions of the configuration spaces just considered. More precisely, let $Z_{p,q}$ denote the space of configurations $\mathbf{z}=\mathbf{z}_{+}\cup \mathbf{z}_{-}$ of distinct punctures on $\pa D$ with $p$ elements in $\mathbf{z}_{+}$ and $q$ in $\mathbf{z}_{-}$. Then $Z_{p,q}$ is a $(p+q)$-dimensional manifold. Recall the space $\mathcal{J}_{\Sigma}$ of admissible almost complex structures on $\C^{n}$ equal to the standard structure in a neighborhood of $f(\Sigma)$. Also, let $\mathcal{D}$ denote the space of admissible displacing Hamiltonians $H\colon \C^{n}\times[0,1]\to\R$ fulfilling the conditions specified in Section \ref{sec:dispham}.

Consider the bundle $\overline\XX_{p,q}\to Z_{p,q}\times\mathcal{J}_{\Sigma}\times\mathcal{D}$ with fiber over $(\mathbf{z},J,H)$ given by $\XX_{p,q}$ defined with respect to the Riemann surface $D_{\mathbf{z}}$ and the almost complex structure $J$.  The space $\overline\XX_{p,q}$ does not depend on the Hamiltonian, and depends on the complex structure only through the condition of holomorphicity at the boundary;  in particular, the bundle  $\overline\XX_{p,q}\to Z_{p,q}\times\mathcal{J}_{\Sigma}\times\mathcal{D}$ is canonically trivialized along the last two factors, since all elements $J\in\mathcal{J}_{\Sigma}$ agree near the boundary.  Given that, one may obtain $C^{1}$-charts giving $\overline\XX_{p,q}$ the structure of a Banach manifold by following \cite[Section 5.6]{EES1}.  Consider also the bundle $\overline\YY_{p,q}\to\overline\XX_{p,q}$ with fiber over $\XX_{p,q}$ equal to $\YY_{p,q}$. Using local trivializations as in \cite[Section 3.2]{EES} we find that $\overline\YY_{p,q}\to\overline\XX_{p,q}$ is a locally trivial bundle with $C^{1}$-transition functions.

\subsubsection{Cauchy-Riemann operators as Fredholm sections}\label{app:CRFredholm}
Consider the bundle $\overline\YY_{p,q}\to\overline{\XX}_{p,q}\times[0,\infty)$, extended trivially over the last factor, and the section $\bar\pa_{\mathrm{F}}\colon\overline{\XX}_{p,q}\to\overline\YY_{p,q}$ where
\[
\bar\pa_{\mathrm{F}}(u,\mathbf{z},J,H,r)=\left(du+X_{H}\otimes\gamma_{r}\right)^{0,1},
\]
where the complex anti-linear part is with respect to the complex structure $(J,i)$, where $X_{H}$ is the Hamiltonian vector field of $H$, and where $\gamma_{r}$ is the $1$-form described in Section \ref{sec:dispham}.  If $\bar\pa_{\mathrm{F}}(u)=0$ we say $u$ is a Floer holomorphic disk. We write $\bar\pa$ for the restriction of $\bar\pa_{\mathrm{F}}$ to $H=0$ and $r=0$.

\begin{Remark}\label{rem:Fourierexpansion}
Since an admissible Hamiltonian $H=\{H_t\}_{t\in [0,1]} \in \mathcal{D}$ is constant near $t=0,1$ and since $\gamma_r$ has compact support, it follows that for each $r_0$ there exists $\delta_0$ such that $\gamma_r\otimes X_H=0$ in a $\delta_0$-neighborhood of $\pa D_{p,q}$. Since in addition $J=J_0$ near $\Sigma$,   \eqref{Eqn:CR-1} reduces to the ordinary $\bar\pa$-equation $\bar\pa_{J_0} u=du^{0,1}=0$ in this $\delta_0$-neighborhood. In particular, this holds in a neighborhood of each puncture and we find that each Floer disk admits a Fourier expansion
\[
u(\tau+i\sigma)=\sum_{k\ge 0} c_k e^{-(k+\frac12)\pi(\tau+i\sigma)},\quad c_k \in \begin{cases} \R^{n}\text{ for all $k$ at a positive puncture} \\   i\R^{n} \text{  for all $k$ at a negative puncture}, \end{cases}
\]  
where $\tau+it\in[0,\infty)\times[0,1]$ conformally parameterizes a neighborhood of one of the punctures of $D_{p,q}$.
\end{Remark}

The section $\bar\pa_{\mathrm{F}}$ and its restriction $\bar\pa$ are both Fredholm sections, with linearization in the $\XX_{p,q}$-direction given by
\begin{equation}\label{eq:lindbar}
D(\bar\pa_{\mathrm{F}})(v)=(\nabla v)^{0,1}+Kv,
\end{equation}
where $v\in T_{u}\XX_{p,q}$ and where $K$ is a compact operator that vanishes on the boundary, see \cite[Lemma 4.2]{EES} and \eqref{eq:defG}. 
\begin{Remark}
We comment on our choice of Sobolev norm. It is natural to require that our configuration spaces consist of continuous maps. Keeping regularity as low as possible this leads to a Sobolev space of maps with one derivative in $L^{p}$, $p>2$. In order to establish necessary elliptic estimates in that set up one uses Calderon-Zygmund techniques as described in e.g.~\cite{McDSal2}. Here, instead, we use the lowest regularity  possible, involving only $L^{2}$-norms, which places us in $\sblv^{2}$. In this setup it is more straightforward to establish elliptic estimates via Fourier analysis in combination with doubling across the boundary. In order for the doubling of a function to lie in $\sblv^{2}$ one must make sure that the second derivative of the doubling lies in $L^{2}$. This is the origin of our ``extra'' boundary condition, where we require $\bar\pa u|_{\pa D}=0$.     
\end{Remark}
The Fredholm index of the operator in \eqref{eq:lindbar} is equal to the Fredholm index of $(\nabla v)^{0,1}$ which in turn was computed in \cite[Proposition 6.17]{EES1}: the Fredholm section $\overline{\XX}_{p,q}\to\overline{\YY}_{p,q}$ has index
\begin{align*}
\ind(D(\bar\pa_{\mathrm{F}}))=2p + (1-q)n.
\end{align*}

The Fredholm section $\bar\pa_{\mathrm{F}}$ is transverse to the $0$-section. Indeed, this is straightfoward to see using variations of the Hamiltonian vector field and unique continuation, see e.g.~\cite[Section 9.2]{McDSal2}. For the operator $\bar\pa$ it is not quite as straightforward: if $p>1$ then solutions may be multiple covers of solutions with fewer punctures and it is impossible to achieve transversality perturbing the almost complex structure only. In order to  deal with such problems a number of more involved pertubation schemes have been developed, e.g.~Kuranishi structures as in \cite{FO3} or polyfolds as in \cite{HWZ}. A common feature of these is that the solution spaces are no longer smooth manifolds but more general objects like Kuranishi spaces or weighted branched manifolds. However, in the cases of relevance for this paper all (unperturbed) holomorphic disks have exactly one positive puncture $p=1$ and in this case transversality can be achieved by  perturbing the almost complex structure only. The argument involves counting the number of branches near the positive puncture to show that cancellations corresponding to the multiply covered case cannot arise, see \cite[Lemma 4.5 (1)]{EES} for the details. This means in particular that the Fredholm theory we need is rather straightforward  and need not appeal to any developments beyond classical Banach bundles. In particular, all the transversality results in Section \ref{Sec:transversality} are standard applications of the Sard-Smale theorem.   

\subsection{Disks with only one positive puncture}\label{ssec:1+gauge}
As explained in Section \ref{Sec:Breaking}, it follows from general transversality results in combination with Gromov-Floer compactness that in order to understand the compactification of the moduli space of closed Floer disks with boundary on the exact Lagrangian $f\colon \Sigma\to \C^{n}$ we need to study, except for the closed disks themselves, only Floer disks with one negative puncture and holomorphic disks with one positive puncture. In fact the Gromov-Floer boundary of the moduli space of closed disks is the product of the two corresponding moduli spaces.    

Theorem \ref{Thm:C1boundary} asserts that the compactified moduli space admits a natural structure as a smooth manifold with boundary. Here we will deal with a preliminary step towards this goal, namely to describe the smooth structures on the factors of the boundary. For the moduli space of Floer disks with a negative puncture this is immediate from the discussion in Section \ref{app:CRFredholm} since in this case the moduli space of solutions is the inverse image of the $0$-section under the Fredholm map $\bar\pa_{\mathrm{F}}$. In order to simplify constructions in later sections we will sometimes use the following slight reformulation of the Floer operator. Recall that we defined the spaces $\XX_{0,1}$ and $\YY_{0,1}$ as Sobolev spaces on the disk with a variable puncture at $\zeta\in\pa D$. Here we define the spaces $\XX_{1}$ and $\YY_1$ as the corresponding spaces but with the puncture fixed at $-1\in\pa D$. For fixed $\zeta\in\pa D$ we let $\gamma_{\zeta,r}$ denote the pull-back of the the $1$-form $\gamma_{r}$ under multiplication by $-\zeta$, viewed as an automorphism of the disk. Then precomposing with such rotations we find that the Floer equation
\begin{equation}\label{eq:floer+rot}
(du+\gamma_{r,\zeta}\otimes X_{H})^{0,1}=0
\end{equation}
on $D_{-1}$ with negative puncture at $-1$ is equivalent to the Floer equation
\[
(du+\gamma_{r}\otimes X_{H})^{0,1}=0
\]
on $D_{\zeta}$ with negative puncture at $\zeta$. We will use the configuration space $\XX_{1}\times\pa D\times[0,\infty)$ for Floer disks with one negative puncture and get the resulting solution space
\[
\FF_{1}=\bar\partial_{\mathrm{F}}^{-1}(0_{\YY_{1}}),
\]
where $\bar\partial_{\mathrm{F}}$ is the operator in \eqref{eq:floer+rot}. Note that precomposition with rotation gives a natural identification with this solution space as defined originally. The general properties of the linearization $D\bar\partial_{\mathrm{F}}$ discussed previously continue to hold  in this slightly modified setup. 

In the case of holomorphic disks with one positive puncture it is less straightforward to define a smooth structure: it is natural to define the moduli space as the quotient of $\bar\pa^{-1}(0)$ by the group $G_{1}$ of conformal automorphisms of the disk $D$ that fix $1\in\pa D$. However, $G_{1}$ does not act with any uniformity on the configuration space since it distorts the underlying Sobolev norm by an arbitrarily large amount, and thus the quotient does not inherit a smooth structure. To equip the moduli space of holomorphic disks with one positive puncture with a smooth structure we instead adopt a gauge fixing strategy, which we describe next.

Write $a_{\pm}$ for the two preimages of the double point in $\Sigma$. Let $S_{\epsilon}(a^{\pm})\subset \Sigma$ denote spheres of radius $\epsilon>0$ around $a_{\pm}$ bounding balls $B_{\epsilon}(a^{\pm})$. We will fix gauge by marking two points $\pm i\in\pa D$ by requiring that our holomorphic disks $u\colon D\to \C^{n}$ with boundary on $f\colon \Sigma\to\C^{n}$ with positive puncture at $1\in\pa D$  satisfy $u(\mp i)\in S_{\epsilon}(a^{\pm})$. Furthermore, to avoid over-counting, we require that the open boundary arc between $1$ and $\pm i$ maps to the interior of $B_{\epsilon}(a^{\mp})$. The resulting maps are stable, and we will see below that their moduli space can be identified with the zero-set of a Fredholm section of a Banach bundle, giving an induced smooth structure.  The main point of the following discussion is to show that the above gauge-fixing requirements do not exclude any solutions. 

We start by considering the rate at which holomorphic disks approach the double point. Recall from Remark \ref{rem:Fourierexpansion} that any (Floer) holomorphic disk with a positive puncture at the double point $a$ has a Fourier expansion
\begin{equation}
u(\tau+i\sigma)=\sum_{k\ge 0} c_k e^{-(k+\frac12)\pi(\tau+i\sigma)}, \quad c_k\in\R^{n}
\end{equation}
in half-strip coordinates around the puncture, which we take to lie at $1\in \pa D$. Our next result shows that generically the first Fourier coefficient is non-zero; we remark that the vanishing or non-vanishing of this coefficient is independent of the choice of half-strip neighborhood,  and is invariant under the reparametrization action of $G_{1}$.

\begin{Lemma}\label{Lem:nondegasympt}
For generic $J\in\JJ_{\Sigma}$, if $u\in\bar\pa^{-1}(0) \subset \XX_{1,0}$ then the leading Fourier coefficient $c_0\in\R^{n}$ of its Fourier expansion near its positive puncture is non-zero.
\end{Lemma}

\begin{proof}
Fix strip coordinates $\tau+it\in[0,\infty)\times[0,1]$ on a neighborhood of the puncture.  Introduce a weight function $w_{\delta}$ such that $w_{\delta}(\tau+it)=e^{\delta|\tau|}$ near the puncture, with $\frac{\pi}{2}<\delta<\frac{3\pi}{2}$.  
The space $\XX_{1,0;\delta}$ is defined in the same way as $\XX_{1,0}$, but using  the weighted Sobolev space $\sblv^{2}_{\delta}(D_{1},\C^{n})$ with weight function $w_{\delta}$ in place of $\sblv^{2}(D_{1},\C^{n})$.  Similarly, we define $\YY_{1,0;\delta}$ in the same way as $\YY_{1,0}$ but replacing the constant weight function with $w_{\delta}$, yielding a Banach bundle $\overline{\YY_{1,0;\delta}}\to \overline{\XX_{1,0;\delta}}$. 

Recalling that the immersion $f$ has standard form with all complex (or K\"ahler) angles equal to $\frac{\pi}{2}$ at the double point, we find that the section $\bar\pa\colon \overline{\YY_{1,0;\delta}}\to\overline{\XX_{1,0;\delta}}$ is Fredholm. However, the positive exponential weight lowers the Fredholm index by $n$, see e.g.~\cite[Proposition 6.16]{EES1}. Hence, noting that the action of $G_{1}$ preserves $\sblv^{2}_{\delta}$, the expected dimension of the corresponding moduli space is $(n-1)-n=-1$, which means that for a generic almost complex structure this problem has no solutions. In terms of ordinary holomorphic disks, we find that for a generic almost complex structure the first Fourier coefficient of any solution is non-zero.
\end{proof}

Fix an almost complex structure as in Lemma \ref{Lem:nondegasympt} and let $u\colon D_{1}\to \C^{n}$ be a holomorphic disk. For small $\epsilon>0$ let $p^{\pm}\in\pa D$ be the points closest to the puncture $1$ such that $u(p^{\pm})\in S_{\epsilon}(a^{\pm})$.

\begin{Lemma}\label{l:smallsphere}
For all sufficiently small $\epsilon>0$ any holomorphic disk $u$ intersects $S_{\epsilon}(a^{\pm})$ transversely at $u(p^{\pm})$.
\end{Lemma}    

\begin{proof}
Consider a sequence of solutions with $u(p^{+})$ not transverse to $S_{\epsilon}(a^{+})$ as $\epsilon\to 0$. As the moduli space is compact we find a limiting solution $u_\infty$. If the first Fourier coefficient of $u_{\infty}$ is non-zero then there exists $\epsilon_0>0$ such that $u_{\infty}$ meets $S_{\epsilon}(a^{+})$ transversely for all $\epsilon\in(0,\epsilon_{0}]$. This contradicts the non-transversality  of the corresponding intersections of the solutions in the sequence, so we conclude that the first Fourier coefficient of $u_{\infty}$ must vanish. This  contradicts Lemma \ref{Lem:nondegasympt}.
\end{proof}

We next control other intersections with $S_{\epsilon}(a^{\pm})$. Consider the evaluation map $\ev\colon \XX_{1,0}\to\Sigma$ at the point $-1\in\pa D$ (note this is the opposite point on the boundary to the puncture which is fixed at $1\in\partial D$) and consider the subset $\ev^{-1}(a^{+})\cup\ev^{-1}(a^{-})$. The Fredholm section restricted to these submanifolds has index $1$, and thus generically there  are finitely many solutions which pass through one of the Reeb chord endpoints. It then follows from Lemma \ref{l:smallsphere} that there exists $\epsilon>0$ such that the intersection points closest to $1$ are transverse and such that the following holds. If $u(p)\in B_{\epsilon}(a^{+})$ (respectively $u(p)\in B_{\epsilon}(a^{-})$) and if there is another intersection point $u(p')\in S_{\epsilon}(a^{+})$ (resp. $u(p')\in S_{\epsilon}(a^{-})$) such that the order of these points and the puncture on $\pa D$ is $1,p',p$ (resp. is $p,p',1$), then some point $p''\in\pa D$ between $p'$ and $p$ (resp. between $p$ and $p'$) maps outside $B_{2\epsilon}(a^{+})$ (resp. outside $B_{2\epsilon}(a^{-})$), i.e. $u(p'')\notin B_{2\epsilon}(a^{\pm})$. 

With this established we now define the smooth structure on the moduli space. Let $\XX_{1,0}'\subset \XX_{1,0}$ denote the configuration space of maps $u$ with $u(\pm i)\in S_{\epsilon}(a^{\mp})$ and such that no point in the arc in $\pa D$ of length $\frac{\pi}{2}$ between $1$ and $\pm i$ maps outside $B_{\epsilon}(a^{\pm})$. The tangent space $T_{u}\XX_{1}'\subset T_{u}\XX_{1}$ is the codimension $2$ subspace of vector fields  tangent to $S_{\epsilon}(a^{\pm})$ at $\mp i$. We now define the moduli space as    
\begin{equation} \label{Eqn:MMsmooth}
\MM=\bar\pa^{-1}(0)\cap \XX'_{1,0}
\end{equation}
and observe that $\MM$ then comes equipped with a smooth structure provided the restriction of $\bar\pa$ to $\XX'_{1,0}$ is transverse to $0$. The argument discussed above, perturbing the almost complex structure near the positive puncture, still works for $\XX_{1,0}'$, and thus we get a smooth structure on $\MM$.  \emph{A priori}, this smooth structure depends on the choice of gauge condition. Using the fact that the $C^{0}$-norm of a holomorphic map controls all other norms, it is not hard to see that different gauge conditions lead to the same smooth structure. This, however, will not play any role in our main application and will therefore not be discussed further.  

\subsection{Further remarks on Gromov-Floer convergence}\label{ssec:gageandconv}
It will be important to have more detailed information on Gromov-Floer convergence of sequences in $\FF_{0}$. Fix small spheres $S_{\epsilon}(a^{\pm})$ as in Section \ref{ssec:1+gauge}. 
For any $M>0$, the open set
\[
U_M=\{u\in\FF_{0}\colon \sup_{z\in D}|du(z)|>M\}
\]  
is a neighborhood of the Gromov-Floer boundary of $\FF_{0}$. 
For $M>0$ and $u\in\FF_{0}$, let $B^{(1)}(u,M)=|du|^{-1}([M,\infty))\subset D$ and let $\delta(u,M)=\diam(B^{(1)}(u,M))$.

\begin{Lemma}\label{Lem:diamblowup} 
Let $\epsilon_M=\sup_{u\in U_M}\delta(u,M)$. Then $\epsilon_M\to 0$ as $M\to\infty$.
\end{Lemma}
\begin{proof}
If not, there is a sequence of disks in $\FF_0$ with holomorphic disk bubbles forming at (at least) two points. This contradicts Lemma \ref{lemma:weakmoduli}.
\end{proof}

It follows from Lemma \ref{Lem:diamblowup} that for each $\delta_0>0$ there is $M>0$ such that for every $u\in U_{M}$, $\delta(u,M)<\delta_0$. Let $I_{u}\subset \pa D$ be the smallest interval containing $B^{(1)}(u,M)\cap\pa D$, and let $\zeta_{u}$ denote the midpoint of $I_{u}$.

\begin{Lemma}\label{Lemma:mpnearsplit}
For all sufficiently large $M>0$, if $u\in U_M$ then there are points $\zeta_{\pm}$ in $I_{u}$ such that $u(\zeta_{\pm})\in S_{\epsilon}(a^{\pm})$. 
\end{Lemma}

\begin{proof}
If the lemma does not hold then there is $\sigma \in \{+,-\}$ and  a sequence of solutions $u_j\in U_{M_{j}}$ with $M_{j}\to\infty$ such that $u_{j}(I_{u_j})\cap S_{\epsilon}(a^{\sigma})=\varnothing$.  Passing to a subsequence, we may assume that $\zeta_{u_j} \rightarrow \zeta$ converge. By Gromov-Floer compactness, a subsequence of $u_{j}$ converges to a solution $\tilde u$ in $\FF_{1}$ on compacts subsets of $D-\zeta$. It follows that the intersection number of $\tilde u(\pa D-\zeta)$ and $S_{\epsilon}(a^{\sigma})$ equals $\pm 1$.  Since the intersection number of any closed curve in $\Sigma$ and $S_{\epsilon}(a^{\sigma})$ equals $0$, we find that there is $\zeta_{\sigma}\in I^{u_j}$ such that $u_{j}(\zeta_{\sigma})\in S_{\epsilon}(a^{\sigma})$ for $j$ large enough. The lemma follows. 
\end{proof}

Let $H$ denote the upper half-plane. Given $u\in U_M$, let $H_{\delta}\subset H$ be a half-disk parameterizing a small neighborhood in $D$ centered on $\zeta_u$. 
Let $\zeta_{\pm}$ be the points in $\pa H_{\delta}$ most distant from the origin satisfying $u(\zeta_{\pm})\in S_{\epsilon}(a^{\pm})$. Let $\alpha=\frac{\zeta_{+}-\zeta_{-}}{2}$ and $\beta=\frac{\zeta_{+}+\zeta_{-}}{2}$. The map 
\[
\psi\colon H\to H, \quad \psi(z)=\alpha z+\beta \quad \textrm{satisfies}  \ \psi(-1)=\zeta_{+}, \ \ \psi(1)=\zeta_{-}.
\]
 Lemma \ref{Lem:diamblowup} implies that $\alpha, \beta\to 0$ as $M\to\infty$. Let $\Omega_{\rho}=\psi^{-1}(H_{\delta})\subset H$, a subset of diameter $\rho = \mathcal{O}(1/\delta)$. 

\begin{Lemma}\label{Lem:limit1}
For any sequence of maps $u_l\in U_{M_l}$ with $M_l\to\infty$, the sequence $u_{l}\circ\psi\colon\Omega_\rho\to\C^{n}$ has a subsequence that converges uniformly on compact subsets to a holomorphic disk $u\colon (H,\pa H)\to(\C^{n},\Sigma)$ with $u(-1)\in S_{\epsilon}(a^{+})$ and $u(1)\in S_{\epsilon}(a^{-})$, and which represents an element in $\MM$. 
\end{Lemma}

\begin{proof}
If the derivative of $u_j\circ\psi$ is uniformly bounded we can extract a convergent subsequence, with limit which must be non-constant since $S_{\epsilon}(a^{+})\cap S(\epsilon)(a^{-})=\varnothing$. The lemma then follows from Lemma \ref{lemma:weakmoduli}. 

To see that the derivative must indeed be uniformly bounded, assume otherwise for contradiction. There is at least one point $\xi\in\pa H$ at which a bubble forms, and  a rescaling argument as above yields a non-constant holomorphic disk. Lemma \ref{lemma:weakmoduli} implies that this is the only bubble and hence $u_j\circ\psi$ must converge to a constant map on compact subsets of $H-\xi$. This however contradicts $S_{\epsilon}(a^{+})\cap S_\epsilon(a^{-})=\varnothing$.
\end{proof}

For $\zeta \in \partial D$, let $B(\zeta;\delta)$ denotes a $\delta$-neighborhood of $\zeta$ in $D$.

\begin{Corollary}\label{Cor:limit2}
For each $\delta>0$ and $\epsilon_0>0$, there exist $M>0$ and $\rho>0$ with the following properties.  For every $u\in U_M\subset\FF_{0}$ there is some $(u_1,u_2)\in\FF_{1}\times\MM$ satisfying:
\begin{enumerate}
\item Let $\zeta\in\pa D$ be  the puncture in domain of $u_1$. The $C^{1}$-distance between $u|_{D-B(\zeta;\delta)}$ and $u_1|_{D-B(\zeta;\delta)}$ is smaller than $\epsilon_0$.
\item Let $\zeta^{\pm}$ denote the intersection points of $u_2(\pa H)$ and $S_{\epsilon}(a^{\pm})$ closest to the puncture.  Consider the unique representative $u_2\colon H\to\C^{n}$ with $u_2(\mp 1) = \zeta^{\pm} \in S_{\epsilon}(a^{\pm})$. The map $u\circ\psi$ is defined on $\Omega_{\rho}$ and the $C^{1}$-distance between $u\circ\psi$ and $u_2|_{\Omega_{\rho}}$ is smaller than $\epsilon_0$.
\end{enumerate}
\end{Corollary}

\begin{proof}
If $(1)$ or $(2)$ do not hold, then we extract a sequence contradicting Lemma \ref{Lem:limit1}. 
\end{proof}

\subsection{Overview}  \label{Ssec:Overview}
Before going into further analytical details we give an overview of the argument underlying Theorem \ref{Thm:C1boundary}. Our main object of study is the $(n+1)$-dimensional moduli space $\FF_0$ of Floer holomorphic maps of the closed disk. The transversality arguments above endow $\FF_0$ with a smooth structure, as the inverse image of the $0$-section of a $C^{1}$ Fredholm section. (By a classical result in differential topology, see e.g.~\cite{Hirsch}, any $C^{1}$-atlas on a finite dimensional manifold contains a uniquely determined compatible $C^{\infty}$-atlas. Hence, in the current set up, regularity of the Fredholm section beyond the $C^{1}$-level does not carry any information on the differential topology of the solution spaces.) 

As we saw in Section \ref{Sec:Breaking}, the Gromov-Floer boundary of $\FF_0$ consists of broken disks, with one component in the compact $1$-manifold $\FF_{1}$ and the other in the in the compact $(n-1)$-manifold $\MM$, joined at the double point. Thus, the Gromov-Floer boundary is simply the product $\FF_1\times\MM$. As usual in Gromov compactness arguments breaking of curves is related to gradient blow up. Recall from Section \ref{ssec:gageandconv} the open subset 
\[
U_M=\left\{u\in\FF_0\colon \sup_{z\in D}|du(z)|>M\right\}
\]
for any $M>0$ and note that $U_M$ is a neighborhood of the Gromov-Floer boundary in the sense that for any $M>0$, $\FF_0-U_M$ is compact, and any sequence $u_j\in\FF_0$ that converges to a broken disk eventually lies in $U_M$.

In order to glue broken disks it is convenient to trade the derivative blow up on the closed disk for a changing domain that undergoes neck-stretching but in which the derivative remains bounded. We thus define a $1$-parameter family of domains $\Delta_{\rho}$ with a neck region of length $2\rho$ in the next section. We express the Floer equation on these domains by identifying the first half of $\Delta_{\rho}$ with the complement of a small neighborhood of $1 \in D_{-1}$, sufficiently small that the $1$-forms $\gamma_{\zeta,r}$ vanish there for all $\zeta \in \partial D$ and for all $r$ in a compact subset of $[0,\infty)$ containing an open neighborhood of the set where there are solutions, and by then extending the Floer operator as the unperturbed Cauchy-Riemann operator over the second half of $\Delta_{\rho}$. These domains also come equipped with natural metrics and associated Sobolev spaces $\sblv^{2}$ comprising the subset of those maps $u$ with $\bar\pa u|_{\pa \Delta_{\rho}}=0$. We still have two marked points $\zeta_{\pm}$ corresponding to $\pm i$ in the second disk, and we still require that these map to $S_{\epsilon}(a^{\pm})$. 

On these domains, with this gauge fixing, we apply Floer's version of Newton iteration to get a smooth embedding of $\FF_{1}\times\MM\times[\rho_0,\infty)$ into $\bar\pa_{\mathrm{F}}^{-1}(0)$. Reparameterization gives a further embedding $\Psi$ into $\FF_{0}$ that we show, using Corollary \ref{Cor:limit2}, covers a neighborhood $U_M$ for some large $M$. This then gives the desired smooth manifold with boundary for the conclusion of Theorem \ref{Thm:C1boundary}, by taking the two pieces $\FF_0-\overline{U_M}$ for $M>0$ large, and $\FF_1\times\MM\times[\rho_0,\rho_1]$ for $\rho_1\gg 0$ sufficiently large, and gluing them together via the diffeomorphism $\Psi$.  

In the remaining sections we explain the details of this construction.

\subsection{Glued domains and configuration spaces}\label{ssec:glueddomains}
We next define a 1-parameter family, parameterised by $[\rho_0,\infty)$, of disks with two boundary marked points. Fix a small neighborhood $B_{\delta}\subset D$ of $-1$ which we think of in terms of upper half plane coordinates around $0$,
\begin{equation} \label{Eqn:HalfPlaneCoord}
h\colon \{w\colon \im(w)\ge 0,\;|z|\le 1\}\to \{z\in D\colon \re(z)\le 0\},\quad
h(w)=\frac{w-i}{w+i}.
\end{equation}
As above, we consider a metric in $D$ in which a neighborhood of $-1$ becomes cylindrical. In the notation of \eqref{Eqn:HalfPlaneCoord} we fix $\delta>0$ and consider coordinates $\tau+it\in(-\infty,0]\times[0,1]$ and the parameterizing map
\[
\tau +it\mapsto \delta e^{\pi(\tau+it)}
\]
and use a metric which is the standard flat metric in these coordinates (where we interpolate to the metric in $D$ over the region $\delta<|w|<2\delta$).

We next consider a once punctured disk with two marked points parameterized by the upper half plane: 
\[
h\colon\{w\colon \im(w)\ge 0\}\to D_1,\quad h(w)=\frac{w-i}{w+i},
\] 
where the marked points $\pm 1$ map to $\mp i$. Here we use strip coordinates $\tau+it\in[0,\infty)\times [0,1]$ on a neighborhood of the puncture at $\infty$,
\[
\tau+it\mapsto r e^{\pi(\tau+it)},
\]
for some $r>0$, and a metric that agrees with the standard metric in the strip.

In order to find a good functional analytic set up near broken disks we will use the coordinates above to join two disks into one by connecting them by a long neck. Write $D^-$ for the disk with half plane coordinates and strip like end near $-1$ and write $D^+$ for the once punctured disk with two marked points and coordinates as above.
For $\rho\in[\rho_0,\infty)$, define the domain $\Delta_\rho$ as the disk obtained from joining
\[
\Delta^{-}_{\rho}=D^{-}-\left((-\infty,-\rho)\times[0,1]\right)\quad\text{ and }\quad
\Delta^{+}_{\rho}=D^{+}-\left((\rho,\infty)\times[0,1]\right)
\]  
across the boundary intervals $\{-\rho\}\times[0,1]$ and $\{\rho\}\times[0,1]$. We take the two points $\zeta_{\pm}\in\pa \Delta_{\rho}$ that correspond to $\mp i\in\pa D^+$ to be marked points. The domain $\Delta_{\rho}$ inherits a natural Riemannian metric from its pieces. In particular there is strip region of length $2\rho$ connecting the two disks in $\Delta_{\rho}$ that we will often identify with, and denote by,  $[-\rho,\rho]\times[0,1]\subset\Delta_{\rho}$. 

\begin{Remark}\label{rem:H_rho}
It will be convenient in Section \ref{Ssec:NhoodBdary} to have a slightly different picture of this domain, that we will call $H_{\rho}$. Remove the half-disk of radius $\delta e^{-\pi\rho}$ from $\{w\colon \im(w)\ge 0,\;|w|\le 1\}$ -- the remaining part of the upper half plane corresponds to $D^{-}$ --   and insert in its place the half disk $\{w\colon \im(w)>0, |w|\le re^{\pi\rho}\}$, corresponding to $D^{+}$, scaling by $\delta r^{-1}e^{-2\pi\rho}$. The intermediate strip region in $\Delta_{\rho}$ now corresponds to a small annular region in $H_{\rho}$. Note that  the upper half plane metric on $H_{\rho}$ differs drastically from the metric on $\Delta_{\rho}$, whence the different notation.
\end{Remark}

We consider the Sobolev spaces induced by the metric on $\Delta_{\rho}$ and, in analogy with Section \ref{ssec:punctureddisks}, we consider the closed subspace
\[
\widehat{\XX}_{\rho}\subset \sblv^{2}(\Delta_{\rho},\C^{n})\times \sblv^{\frac32}(\pa\Delta_{\rho},\R)
\]
of maps that take the boundary into the Legendrian lift $\tilde f(\Sigma)\subset\C^{n}\times\R$, which are holomorphic on the boundary, $\bar\pa u|_{\pa\Delta_{\rho}}=0$, and which satisfy the marked point condition that $u(\zeta_{\pm})\in S_{\epsilon}(a^{\pm})$. Similarly, we consider the subspace
\[
\widehat{\YY}_{\rho}\subset\sblv^{1}\left(\Hom^{0,1}(T\Delta_{\rho},T\C^{n})\right)
\]
of $(J,i)$-complex anti-linear maps $A$ that vanish on the boundary, $A|_{\pa\Delta_{\rho}}=0$. 
The families $\widehat{\XX}_{\rho}$ and $\widehat{\YY}_{\rho}$ naturally form bundles $\widehat{\XX}$ and $\widehat{\YY}$ over $[\rho_0,\infty)$.

\begin{Lemma} \label{Lem:C1source}
$\widehat{\XX}$ is a locally trivial $C^1$-smooth Banach bundle.
\end{Lemma}

\begin{proof}
We trivialize using precomposition with diffeomorphisms. Consider the middle strip $[-\rho,\rho]\times[0,1]\subset \Delta_{\rho}$, where we think of $[-\rho,0]\times[0,1]$ and $[0,\rho]\times[0,1]$ as subsets of $\Delta^+_{\rho}$ and $\Delta^{-}_{\rho}$, respectively. Let $\rho=\rho'+c$ and define the map $\tilde\psi_{\rho',\rho}\colon [-\rho',\rho']\times[0,1]\to [-\rho,\rho]\times[0,1]$ as
\[
\tilde\psi_{\rho',\rho}(\tau+it)=
\begin{cases}
\tau+it - c,  &\text{ for } -\rho'<\tau<-\rho'+5,\\
a(\tau)\tau + it -c(\tau), &\text{ for }-\rho'+5\le \tau<-\rho'+10,\\
a\tau +it, &\text{ for } -\rho'+10\le \tau < \rho'-10,\\
a(-\tau)\tau + it +c(-\tau), &\text{ for }\rho'-10\le \tau<\rho'-5,\\
\tau+it + c,  &\text{ for } \rho'-5\le\tau<\rho'.
\end{cases}
\]
Here $c(\tau)$, $a(\tau)$ are  smooth cut-off functions defined on $[-\rho'+5, -\rho'+10]$ (constant near the end-points); $c(\tau)$  decreases from $c$ to $0$ and $a(\tau)$ increases from $1$ to $a$ (where $a$ is the unique stretch factor which ensures that $\tilde\psi_{\rho',\rho}$ has the correct image).  Note that if the supports of the derivatives of $a(\tau)$ and $c(\tau)$ are chosen disjoint, $\tilde{\psi}_{\rho',\rho}$ is a diffeomorphism that agrees with translation by $\pm c$ near the ends of the strip region.

Let $\psi_{\rho',\rho}$ be a diffeomorphism that approximates $\tilde{\psi}_{\rho',\rho}$, but which is holomorphic on the boundary (it is straightforward to construct such an approximation using Taylor expansion near the boundary). Finally, define the diffeomorphism
\[
\Psi_{\rho',\rho}\colon \Delta_{\rho'}\to\Delta_{\rho}
\]
as follows:
\[
\Psi_{\rho',\rho}(z)=
\begin{cases}
z &\text{ for } z\in\Delta^{+}_{\rho'}-([-\rho',0]\times[0,1]),\\
\psi_{\rho',\rho}(z) &\text{ for }z\in[-\rho',\rho']\times[0,1],\\
z &\text{ for } z\in\Delta^{-}_{\rho'}-([0,\rho']\times[0,1]).
\end{cases}
\]
We define the trivialization over $[\rho_0,\infty)$ by precomposing with $\Psi_{\rho',\rho}$.  More precisely, the diffeomorphism $\Psi_{\rho',\rho}$ induces a diffeomorphism of Sobolev spaces $\sblv^{2}(\Delta_{\rho'},\C^{n}) \rightarrow \sblv^{2}(\Delta_{\rho},\C^{n})$, similarly for $ \sblv^{\frac32}(\pa\Delta_{\rho},\R)$; noting that the maps $\psi_{\rho',\rho}$ are holomorphic on the boundary, these diffeomorphisms further preserve the subspaces  $\widehat{\XX}_{\rho'}$, $\widehat{\XX}_{\rho}$. The induced diffeomorphisms of those  spaces depend $C^1$-smoothly on $\rho', \rho$ provided these vary only in bounded intervals (which ensures that the derivatives of the maps $\psi_{\rho',\rho}$ may be taken uniformly bounded).  Covering $[\rho_0,\infty)$ by a countable collection of intervals of length $1$, precomposition with appropriate $\Psi_{\rho',\rho}$ yields local trivialisations of the bundle $\widehat{\XX}$ which define a $C^1$-atlas. 
\end{proof}

In order to adapt the current set up to Floer's Picard lemma we introduce 
an ``exponential map'' that gives local coordinates on $\widehat{\XX}$. More precisely, we consider $\widehat{\XX}$ as a bundle over $[\rho_0,\infty)$ with fiber $\widehat{\XX}_{\rho}$.
We first describe the tangent space of $\widehat{\XX}_{\rho}$. As in Lemma \ref{lemma:exponentialmap}, the tangent space $T_{u}\widehat{\XX}_{\rho}$ at a map $u$ is the space of vector fields $v$ along $u$ in $\sblv^{2}(\pa\Delta_{\rho},\C^{n})$, tangent to $f(\Sigma)$ along the boundary, such that $(\nabla v)^{0,1}|_{\pa\Delta_{\rho}}=0$, and such that the component of $v$ at $\zeta_{\pm}$ perpendicular to $S_{\epsilon}(a^{\pm})$ vanishes. We write $\EE(u)$ for this space of vector fields with the $\sblv^{2}$-norm. 

With respect to the metric $\hat{g}$ of Section \ref{ssec:metric}, there is an exponential map $\exp_{u}\colon\EE(u)\to\widehat{\XX}_{\rho}$. The image of a small ball in $\EE(u)$ under $\exp_u$ gives a $C^{1}$-chart on $\widehat{\XX}_{\rho}$ near $u$.  To see how such charts vary with $\rho \in [\rho_0, \infty)$, consider precomposition with the diffeomorphism $\Psi=\Psi_{\rho',\rho}\colon \Delta_{\rho'}\to \Delta_{\rho}$ from  Lemma \ref{Lem:C1source}.   The exponential map above is equivariant under precomposition with $\Psi$:
\begin{equation} \label{Eqn:Exp}
\exp_{u\circ\Psi}(v\circ\Psi)=\left(\exp_{u}(v)\right)\circ\Psi.
\end{equation}
We point out that after fixing the domain for a local trivialization, the linearized variation in the $\rho$-direction in the direction of increasing $\rho$ is a vector field that corresponds to moving the two marked points $\zeta_{\pm}$ toward each other.  

\subsection{Floer-Picard lemma} \label{Ssec:Floer-Picard}
Our main technical tool is \cite[Lemma 6.1]{EkholmSmith}, which is a version of Floer's Picard lemma \cite{Floer:mem} that we restate and prove here for convenience. Let $T$ and $M$ be finite dimensional smooth manifolds and let $\pi_{X}\colon X\to M\times T$ be a smooth bundle of Banach spaces over $M\times T$. Let $\pi_{Y}\colon Y\to T$ be a smooth bundle of Banach spaces over $T$. Let $f\colon X\to Y$ be a smooth bundle map of bundles over $T$ and write $f_{t}\colon X_{t}\to Y_{t}$ for the restriction of $f$ to the fiber over $t\in T$ (where the fiber $X_t$ is the bundle over $M$ with fiber $X_{(m,t)}$ at $m\in M$). If $m\in M$, then write $d_{m}f_{t}\colon X_{(m,t)}\to Y_{t}$ for the differential of $f_{t}$ restricted to the vertical tangent space of $X_{t}$ at $0\in X_{(m,t)}$, where this vertical tangent space is identified with the fiber $X_{(m,t)}$ itself; similarly, the tangent spaces of $Y_{t}$ are identified with $Y_{t}$ using linear translations. Denote by $0_{Y}$ the $0$-section in $Y$,  and by $D(0_X;\epsilon)$  an $\epsilon$-disk sub-bundle of $X$.

\begin{Lemma}\label{Lem:FloerPicard}
Let $f\colon X\to Y$ be a smooth Fredholm bundle map of $T$-bundles, with Taylor expansion in the fiber direction:
\begin{equation}\label{e:TaylorFP}
f_{t}(x)=f_{t}(0)+d_{m}f_{t}\,x+N_{(m,t)}(x),\quad\text{where }\; 0,x\in X_{(m,t)}.
\end{equation}
Assume that $d_{m}f_{t}$ is surjective for all $(m,t)\in M\times T$ and has a smooth family of uniformly bounded right inverses $Q_{(m,t)}\colon Y_t\to X_{(m,t)}$, and that the non-linear term $N_{(m,t)}$ satisfies a quadratic estimate of the form
\begin{equation}\label{e:QuadraticFP}
\|N_{(m,t)}(x)-N_{(m,t)}(y)\|_{Y_t}\le C\|x-y\|_{X_{(m,t)}}(\|x\|_{X_{(m,t)}}+\|y\|_{X_{(m,t)}}),
\end{equation}
for some constant $C>0$.  Let  $\ker(df_{t})\to M$ be the vector bundle with fiber over $m\in M$ equal to $\ker(d_mf_{t})$. 
If $\|Q_{(m,t)}f_{t}(0)\|_{X_{(m,t)}}\le\frac{1}{8C}$, then for $\epsilon<\frac{1}{4C}$,
$f^{-1}(0_{Y})\cap D(0_{X};\epsilon)$  is a smooth submanifold
diffeomorphic to the bundle over $T$ with fiber at $t\in T$ the $\epsilon$-disk bundle in $\ker(df_{t})$.
\end{Lemma}

\begin{proof}

The proof for the case when $M$ and $T$ are both points appears in Floer \cite{Floer:mem} and generalizes readily to the case under study here. We give a short sketch pointing out some features that will be used below. Let $K_{(m,t)}=\ker(d_mf_{t})$ and choose a smooth splitting $X_{(m,t)}=X_{(m,t)}'\oplus K_{(m,t)}$ with projection $p_{(m,t)}\colon X_{(m,t)}\to K_{(m,t)}$. For $k_{(m,t)}\in K_{(m,t)}$, define the bundle map $\widehat{f}_{(m,t)}\colon X_{(m,t)} \to Y_{t}\oplus K_{(m,t)}$, 
\[
\widehat{f}_{(m,t)}(x)=\bigl(f_t(x)\;,\;p_{(m,t)}\,x-k_{(m,t)}\bigr). 
\]
Then solutions to the equation $f_t(x)=0$, $x\in X_{(m,t)}$ with $p_{(m,t)}\,x=k_{(m,t)}$ are in one-to-one correspondence with solutions to the equation $\widehat{f}_{(m,t)}(x)=0$. Moreover, the differential $d\widehat{f}_{(m,t)}$ is an isomorphism with inverse 
$\widehat{Q}_{(m,t)}\colon Y_t\oplus K_{(m,t)}\to X_{(m,t)}$, 
\[
\widehat{Q}_{(m,t)}\,(y,k)= Q_{(m,t)}\,y\; +\;k.
\]
On the other hand, solutions of the equation $\widehat{f}_{(m,t)}(x)=0$ are in one-to-one correspondence with fixed points of the map $F_{(m,t)}\colon X_{(m,t)}\to X_{(m,t)}$ given by
\[
F_{(m,t)}(x)=x\;-\;\widehat{Q}_{(m,t)}\,\widehat{f}_{(m,t)}(x).
\]
Fixed points are obtained from the Newton iteration scheme: if
\[
v_0=k_{(m,t)},\quad v_{j+1}=v_j\;-\;\widehat{Q}_{(m,t)}\,\widehat{f}_{(m,t)}(v_j),
\]
then $v_{j}$ converges to $v_\infty$ as $j\to\infty$ and $F_{(m,t)}(v_\infty)=v_{\infty}$. 
Note that these fixed points correspond to transverse intersections between the graph of $\widehat{Q}_{(m,t)}$ and the diagonal $\Delta_{(m,t)}\subset X_{(m,t)}\times X_{(m,t)}$ and consequently form a smooth bundle first over the kernel bundle over $M$ for fixed $t$ and then over $T$.

Furthermore, if $\|f_{t}(0_{X_{(t,m)}})\|$ is sufficiently
small then there is $0<\delta<1$ such that:
\[
\|v_{j+1}-v_j\|\le \delta^{j}\|f_{t}(0)\|
\]
and consequently
\begin{equation}\label{Eq:iterest}
\|v_\infty-v_0\|\le \kappa \|f_{t}(0)\|,
\end{equation}
where $\kappa$ is a constant.
\end{proof}

\subsection{Pre-gluing} \label{Ssec:Pre-gluing}
The first step in our gluing construction is to embed a collar on the Gromov-Floer boundary of $\FF_0$ in the configuration space $\widehat{\XX}\times[0,\infty)$ approximating broken curves. More precisely, in the notation of Section \ref{Sec:transversality}, let 
\[
\NN=\FF_{1}\times\MM\times[\rho_0,\infty),
\]
and let $\NN_{\rho} = \FF_{1}\times\MM$, thought of as the fiber over $\rho$ of the projection $\NN\to[\rho_0,\infty)$. We will define a fiber-preserving embedding
\[
\Pre\colon\NN\to\widehat{\XX}\times\partial D\times[0,\infty),\quad \Pre(\NN_\rho)\subset\widehat{\XX}_\rho\times\partial D\times[0,\infty).
\]

Let $u^-\colon D^{-}\to\C^{n}$ be a map in $\FF_1$, and let $u^+\colon D^+\to\C^{n}$ be a map in $\MM$ that respects marked point conditions as in Sections \ref{ssec:1+gauge} and \ref{ssec:glueddomains}. By Remark \ref{rem:Fourierexpansion}, and by compactness of $\FF_1$ and $\MM$, we find that for each $\epsilon_0>0$, there exists $\rho_0>0$ such that $u^-$ takes the neighborhood $(-\infty,-\rho_0]\times[0,1]$ of the puncture at $-1$ into $B_{\epsilon_0}$ and $u^+$ takes the neighborhood $[\rho_0,\infty)\times[0,1]$ of the puncture at $1$ into $B_{\epsilon_0}$. Assume that $\epsilon_0>0$ is sufficiently small so that $f(\Sigma)$ is standard in $B_{\epsilon_0}$. 

Let $\beta\colon\R\times[0,1]\to\C$ be a smooth function that equals $0$ on $(-\infty,-1]\times[0,1]$, equals $1$ on $[1,\infty)\times[0,1]$, and which is real and holomorphic along the boundary of the domain. (Such a function may be obtained by modifying any suitable cut-off function with real boundary values by addition of a pure imaginary function supported near the boundary and with suitable normal derivative.) For $\rho>\rho_0$ define the map $u^-\#_\rho u^+\colon \Delta_{\rho}\to\C^{n}$
\begin{equation}\label{eq:interpol}
u^-\#_\rho u^+(z)=
\begin{cases}
u^-(z), &\text{for }z\in D^{-}-((-\infty,-\rho]\times[0,1]),\\
(1-\beta(z))u^-(z)+\beta(z)u^+(z) &\text{for }z\in [-\rho,\rho]\times[0,1],\\
u^+(z), &\text{for }z\in D^+-([\rho,\infty)\times[0,1]).
\end{cases}
\end{equation}

We then define the map $\Pre\colon\NN\to\widehat{\XX}\times\partial D\times[0,\infty)$ as 
\begin{equation} \label{Eqn:PreglueRho}
\Pre(u^-,u^+,\rho)=\left(u^-\#_{\rho} u^{+},\zeta(u^{-}),r(u^{-})\right),
\end{equation}
where $(\zeta(u^{-}),r(u^{-}))\in\partial D\times[0,\infty)$ is the coordinate of the Hamiltonian term at $u^{-}$, i.e.~$u^{-}$ solves the Floer equation $\left(du^{-}+\gamma_{\zeta(u^{-}),r(u^{-})}\otimes X_H(u^{-})\right)^{0,1}=0$. By compactness of the moduli spaces involved we find that if $\rho_{0}$ is sufficiently large then $\Pre$ is a fiber preserving embedding (as a map of bundles over $[\rho_0,\infty)$). More precisely, the restriction 
\[
\Pre_{\rho}=\Pre|_{\NN_{\rho}}\colon \NN_{\rho}\to\widehat{\cfig}_{\rho}\times\partial D\times[0,\infty)
\] 
is a family of embeddings which depends smoothly on $\rho$.

Consider the normal bundle $N\NN_{\rho}$ of $\NN_{\rho}=\Pre_{\rho}(\NN_{\rho})\subset\widehat{\cfig}_{\rho}\times\partial D\times[0,\infty)$ as a sub-bundle of the restriction $T_{\NN_{\rho}}\left(\widehat{\cfig}_{\rho}\times\partial D\times[0,\infty)\right)$ of the tangent bundle of $\widehat{\cfig}_{\rho}\times\partial D\times[0,\infty)$ to $\NN_{\rho}$. For simpler notation, write $w_\rho=u^{-}\#_\rho u^{+}$ and recall that $\EE(w_\rho)$ denotes the fiber of the tangent bundle $T_{w_\rho}\widehat{\XX}_{\rho}$, see the last two paragraphs of  Section \ref{ssec:glueddomains}.  

The fiber $N_{w_{\rho}}\NN_{\rho}$ of the normal bundle at $w_{\rho}$ is the $L^{2}$-complement of the subspace
\[
d\Pre(T_{(u^{-},u^{+})}\NN_{\rho})\subset T_{(w_\rho,\zeta,r)}\widehat{\cfig}_{\rho}\times\partial D\times[0,\infty),
\]
where $(\zeta,r)=(\zeta(u^{-}),r(u^{-}))$ are the coordinates of the Hamiltonian term in the Floer equation at $u^{-}$. The $L^{2}$-pairing is defined as
\begin{equation} \label{Eqn:L2pair}
\left\langle (v,\delta\zeta,\delta r) \, , \, (v',\delta\zeta',\delta r')\right\rangle = \langle v \, ,\,  v'\rangle_{\EE(w_\rho)} + \langle  \delta\zeta\, , \, \delta\zeta'\rangle_{T_{\zeta}\partial
 D}+\langle \delta r \, , \, \delta r'\rangle_{T_{r}[0,\infty)},
\end{equation}
where the first summand is the $L^{2}$-pairing on the Sobolev space $\EE(w_\rho)$ induced by the $L^{2}$-pairing on the ambient space, whilst the second and third are the standard inner products on the $1$-dimensional tangent spaces indicated.
 
Let $\ker(u^-)\subset T_{u^{-}}\XX_1\times\partial D\times[0,\infty)$ denote the kernel of the linearized operator $D\bar\pa_{\mathrm{F}}$ at $u^{-}$. An element $v^{-}$ in $\ker(u^{-})$ is a linear combination 
\begin{equation} \label{Eqn:VectorFieldPieces}
v^{-}=v_{\XX_1}^{-}+v_{\pa D}^{-}+v_{[0,\infty)}^{-},
\end{equation}
where $v_{\XX_{1}}^{-}$ is a vector field in $\sblv^{2}(D^{-},\C^{n})$ which satisfies the usual boundary conditions, where $v_{\pa D}^{-}$ is tangent to $\pa D$ at $\zeta(u^{-})$, and $v_{[0,\infty)}^{-}$ is tangent  to $[0,\infty)$ at $r(u^-)$.  Similarly, let $\ker(u^{+})\subset T_{u^{+}}\XX_1'$ denote the kernel of the linearized operator $D\bar \pa$ at $u^{+}$;  elements $v^{+}$ in $\ker(u^{+})$  are vector fields in $\sblv^{2}(D^{+},\C^{n})$ which satisfy the usual boundary conditions and conditions at marked points.  Both linearized operators agree near the punctures with the standard (constant co-efficient) $\bar\pa$ operator, with linear Lagrangian boundary conditions $\R^{n}$ along one boundary component adjacent to the puncture and $i\R^{n}$ along the other. Thus, by Remark \ref{rem:Fourierexpansion}, the vector fields $v^{-}_{\XX_{1}}$ and $v^{+}$ admit Fourier expansions of the same form as the Floer solutions near the punctures.

Let the cut-off function $\beta\colon \R\times[0,1]\to\C$ be as defined previously, in particular be holomorphic along the boundary, and associate elements $\tilde v^{\pm}\in \EE(w_\rho)\oplus T\partial D\oplus T[0,\infty)$ to $v^{\pm}$ as follows. First, $\tilde v^{+}$ is the following vector field along $\Delta_{\rho}$ in $\EE(w_{\rho})$:
\[
\tilde v^{+}(z)=
\begin{cases}
0 &\text{for }z\in D^- -((-\infty,0]\times[0,1]),\\
\beta(z)v^{+}(z) &\text{for }z\in[-\rho,\rho]\times[0,1],\\
v^{+}(z) &\text{for }z\in D^+-([0,\infty)\times[0,1]).
\end{cases}
\]
Analogously, we define the vector field $\tilde v_{\XX_{1}}^{-}$ as follows
\[
\tilde v_{\XX_{1}}^{-}(z)=
\begin{cases}
v_{\XX_{1}}^{-}(z) &\text{for }z\in D^{-}-(-\infty,0]\times[0,1],\\
(1-\beta(z))v_{\XX_1}^{-}(z) &\text{for }z\in[-\rho,\rho]\times[0,1],\\
0 &\text{for }z\in D^{+}-[0,\infty)\times[0,1],
\end{cases}
\]
and then define
\[
\tilde{v}^{-}=
\tilde v_{\XX_1}^{-}+v_{\partial D}^{-}+ v_{[0,\infty)}^{-}\in\EE(w_{\rho})\oplus T_{\zeta}\partial D\oplus T_r[0,\infty).
\] 

\begin{Remark}
We have constructed the vector fields $\tilde{v}^+$ and $\tilde v_{\XX_{1}}^{-}$ only when starting with smooth (by elliptic regularity) vector fields $v^{\pm}$, which makes it manifest that they depend smoothly on $\rho$. In fact, for $\rho$ large enough the cut-off used to construct $\tilde{v}^+$ and $\tilde v_{\XX_{1}}^{-}$ from $v^{\pm}$ is concentrated in the region where the Lagrangian boundary conditions are linear and the almost complex structure is standard, so the initial vector fields $v^{\pm}$ are even analytic in the region of interpolation. More generally, starting with any $\sblv^2$-vector fields $v^{\pm}$ (that are holomorphic on the boundary), we find, since their norms are finite and since the $\sblv^{2}$-norm controls the supremum norm, that there are neighborhoods of the punctures that map into the region where the Lagrangian boundary conditions are linear and the almost complex structure is standard. Hence for $\rho$ large enough the $\rho$-dependence of $\tilde{v}^{+}$ and $\tilde{v}_{\XX_1}^{-}$, viewed as sections in $\EE(w_{\rho})$, would still be smooth  (although, in that case the resulting vector fields would themselves only be $\sblv^2$).
\end{Remark}

Recall that $\FF_1\subset \XX_1\times \partial D \times[0,\infty)$ is a $1$-manifold;   let $v_1^{-}$ be a basis vector in $T_{u^{-}}\FF_1\subset T_{u^{-}}(\XX_1\times\partial D\times[0,\infty))$, i.e.~$v_1^{-}$ is a basis in $\ker(u^{-})$. Also $\MM\subset \XX_{1,0}'$ is an $(n-1)$-manifold. Let $v_{1}^{+},\dots,v_{n-1}^{+}$ be a basis in $T_{u^{+}}\MM\subset T_{u^{+}}\XX_{1}'$, i.e.~in $\ker(u^{+})$. Let $\widetilde{\ker}(u^{-})\subset T_{w_\rho}(\widehat{\XX}\times\partial D\times[0,\infty))$ be the subspace spanned by $\tilde v^{-}$ and let $\widetilde{\ker}(u^{+})\subset T_{w_\rho}(\widehat{\XX}\times\partial D\times[0,\infty))$ denote the subspace spanned by $\tilde{v}_{1}^+,\dots,\tilde{
v}_{n-1}^+$. Note that $L^{2}$-projection gives an isomorphism 
\begin{equation} \label{Eqn:L2projiso}
N_{w_{\rho}}\NN_{\rho}\to \left(\widetilde{\ker}(u^{-})\oplus\widetilde{\ker}(u^{+})\right)^{\perp},
\end{equation}
where $W^\perp$ denotes the $L^{2}$ orthogonal complement of the subspace $W$, which gives a bundle isomorphism that varies uniformly continuously with $\rho$. Below we will often represent the fiber bundle $N\NN_{\rho}\subset T_{\NN_{\rho}}(\widehat{\XX}\times\partial D)$ as the bundle with fiber $\left(\widetilde{\ker}(u^{-})\oplus\widetilde{\ker}(u^{+})\right)^{\perp}$ over $w_{\rho}=u^{-}\#_{\rho}u^{+}$. 

\begin{Remark}
In fact, the uniformly continuous dependence of the bundle isomorphism \eqref{Eqn:L2projiso} on $\rho$ is even uniformly smooth in the following sense. In order to study derivatives at $\rho_0$, we trivialize the bundle $\widehat{\XX}$ over the interval $(\rho_0-1,\rho_0+1)$, as in Lemma \ref{Lem:C1source}, yielding a corresponding family of $L^{2}$-projections for $\rho\in(\rho_0-1,\rho_0+1)$. One can directly check that the derivatives of this family with respect to $\rho$ at $\rho_0$, calculated in the given trivialization, are uniformly bounded as $\rho_0\to\infty$. (The  essential point is that, in the notation of Lemma \ref{Lem:C1source}, if $\rho_0' = \rho_0+1$ then both the translation term $c\in[-1,1]$ and the stretch factor $a \in [(\rho_0-1)/\rho_0, (\rho_0+1)/\rho_0]$ are bounded, which enables one to control the first term in the $L^2$-norm \eqref{Eqn:L2pair} of elements in the image.) Our subsequent argument does not require this uniform smoothness, however, but only the uniform bounds established in Lemma \ref{Lem:partialglu1} below.
\end{Remark}

Using the exponential map introduced before Equation \eqref{Eqn:Exp},  we parameterize a neighborhood of $\NN_\rho\subset\widehat{\XX}_{\rho}\times\partial D\times[0,\infty)$ by a neighborhood of the $0$-section in the restriction of the tangent bundle $T_{\NN_{\rho}}(\widehat{\XX}_{\rho}\times\partial D\times[0,\infty))$, which we take to be the $\epsilon$-disk sub-bundle $T_{\NN_{\rho};\epsilon}(\widehat{\XX}_{\rho}\times\partial D\times[0,\infty))$. The Floer equation then gives smooth maps $f_{\rho}\colon T_{\NN_{\rho};\epsilon}(\widehat{\cfig}_{\rho}\times\partial D\times[0,\infty))\to\widehat{\tcfig}_{\rho}$, defined as follows. 
If $w_{\rho}= u^{-}\#_{\rho} u^{+}\in \NN_{\rho}$ then for $(v,\delta\zeta,\delta r)\in T_{w_{\rho};\epsilon}(\widehat{\cfig}_{\rho}\times\partial D\times[0,\infty))$:
\begin{equation}\label{Eq:Floermaponglued}
f_{\rho}(v,\delta \zeta,\delta r)=\left(
d \exp_{w_{\rho}}(v) + \gamma_{\zeta(u^-)e^{i\delta\zeta},r(u^-)+\delta r}\otimes X_{H}(\exp_{w_{\rho}}(v,c))
\right)^{0,1}.
\end{equation}
This family of maps varies smoothly with $\rho$ and thus gives a smooth bundle map 
\[
f\colon T_{\NN}(\widehat{\XX}\times\partial D\times[0,\infty))\to\widehat\YY
\]
of bundles over $[\rho_0,\infty)$, where the fibers at $\rho\in[\rho_0,\infty)$ of the source and target are respectively $T_{\NN_{\rho}}(\widehat{\XX}_{\rho}\times\partial D\times[0,\infty))$ and $\widehat\YY_{\rho}$.

To simplify notation, we will denote by $f$  the restriction $f|_{N_{\epsilon}\NN}$, i.e.~the fiberwise restriction of $f_{\rho}$ to the radius $\epsilon$ disk bundle $N_{\epsilon}\NN_{\rho}$ in the normal bundle of $\NN_{\rho}\subset\widehat{\XX}_{\rho}\times\partial D\times[0,\infty)$.

\subsection{The gluing map} \label{Ssec:gluing}
To construct the gluing map 
\[
\Phi\colon \NN\to \FF_{0}
\]
which parametrizes a $C^{1}$-neighborhood of the Gromov-Floer boundary we will apply Lemma \ref{Lem:FloerPicard} to the map $f\colon N\NN\to\widehat{\tcfig}$ defined in \eqref{Eq:Floermaponglued}.  
This requires us to verify that the assumptions of Lemma \ref{Lem:FloerPicard} are fulfilled for $f$. Before we go into verifying that we present a dictionary between maps and spaces considered here and those in Lemma \ref{Lem:FloerPicard}.

As above, for $(u^{-},u^{+},\rho)\in\NN$, let $w_{\rho}=u^{-}\#_{\rho} u^{+}$.   In the notation of Lemma \ref{Lem:FloerPicard}:

\begin{itemize}
\item  $[\rho_0,\infty)$ corresponds to $T$; $\FF_{1}\times\MM\approx \NN_{\rho}$ corresponds to $M$;
\item  $N\NN$, with fiber $N_{w_{\rho}}\NN_{\rho}$, corresponds to $X$; $\widehat{\tcfig}$ corresponds to $Y$; 
\item  we use the norm $\|\cdot\|_{N\NN_{\rho}}$ inherited from $T(\widehat{\XX}_{\rho}\times\partial D\times[0,\infty))$, and the usual Sobolev $1$-norm as $\|\cdot\|_{\widehat{\tcfig}_{\rho}}$;
\item the fiber $N_{w_{\rho}}\NN_{\rho}$ of the normal bundle $N\NN_{\rho}$ at $w_{\rho}$ corresponds to $X_{(m,t)}$;
\item $w_{\rho}$ corresponds to $0\in X_{(m,t)}$.
\end{itemize}

We first show that for $\rho\in[\rho_0,\infty)$, $w_{\rho}$ is a good approximation of a solution. In fact the error converges exponentially fast to $0$.

\begin{Lemma}\label{Lem:almostholomorphic}
For $(u^{-},u^{+}) \in \FF_{1}\times\MM$, let $w_{\rho}=u^{-}\#_{\rho} u^{+}\in\NN_{\rho}$. Then as $\rho \to \infty$, 
\[
\sup_{\NN_{\rho}}\|f_{\rho}(w_{\rho})\|_{\widehat{\tcfig}_{\rho}}=\Ordo(e^{-\frac{\pi}{2}\rho}).
\]
\end{Lemma}

\begin{proof}
This is an immediate consequence of the Fourier expansions of $u^{-}$ and $u^{+}$ near the puncture, cf. Remark \ref{rem:Fourierexpansion}, the compactness of $\FF_{1}\times\MM$, and the definition of $u^{-}\#_{\rho} u^{+}$ in  \eqref{eq:interpol}.
\end{proof}

\begin{Lemma}\label{Lem:partialglu1}
There is $\rho_0>0$ such that for $\rho>\rho_0$, the vertical differential 
\[
d_{w_{\rho}}f_{\rho}\colon N_{w_{\rho}}\NN_{\rho}\to \widehat{\tcfig}_{\rho}
\]
is surjective, so there is a bundle map $\widehat{\tcfig}_{\rho}\to N_{w_{\rho}}\NN_{\rho}$ which is a right inverse of $d_{w_{\rho}}f_{\rho}$ on fibers. Furthermore, the inverse map depends smoothly on $\rho$ and is uniformly bounded over $[\rho_0,\infty)$. 
\end{Lemma}

\begin{proof}
This relies on a standard linear gluing argument, as follows. Recall from \eqref{Eqn:L2projiso} the representation of $N_{w_{\rho}}\NN_{\rho}$ as the $L^{2}$ complement
\[
\left(\widetilde{\ker}(u^-)\oplus\widetilde{\ker}(u^+)\right)^{\perp}\subset T_{w_{\rho}}(\widehat{\XX}_{\rho}\times\partial D\times[0,\infty)).
\]
We show that there is $\rho_0>0$ and a constant $C>0$ such that for any $w_{\rho}$ with $\rho>\rho_0$ and any $v_{\rho}\in N_{w_{\rho}}\NN_{\rho}$, the following estimate holds:
\begin{equation}\label{eq:garding}
\|v_{\rho}\|_{N\NN_{\rho}}\leq C\|(d_{w_{\rho}}f_{\rho})v_{\rho}\|_{\widehat{\YY}_{\rho}}
\end{equation} 

Assume that \eqref{eq:garding} is false; then there is a sequence of pairs $(w_{\rho},v_{\rho})$, $\rho\to\infty$, where $w_{\rho}=u^{-}\#_{\rho} u^{+}$ for some $(u^{-},u^{+})\in \FF_{1}\times\MM$ (depending on $\rho$) and $v_{\rho}\in N_{w_{\rho}}\NN_{\rho}$ such that 
\[
\|v_{\rho}\|_{N\NN_{\rho}}=1,\quad \|(d_{w_{\rho}}f_\rho) v_{\rho}\|_{\widehat{\YY}_{\rho}}\to 0 \text{ as }\rho\to\infty.
\] 
By compactness of $\FF_{1}\times\MM$, we pass to a subsequence for which $(u^{-},u^{+})$ converges as $\rho\to\infty$. Write
\[
v_{\rho}=v_{\rho}^{\widehat{\XX}}+v_{\rho}^{\partial D}+v_{\rho}^{[0,\infty)},
\]
subdividing $v_{\rho}$ into components corresponding to the factors of the ambient space $\widehat{\XX}_{\rho}\times\partial D\times[0,\infty)$ of $\NN_{\rho}$. Let $\alpha\colon\Delta_{\rho}\to\C$ be a cut off function similar to $\beta$ in \eqref{eq:interpol} but with  larger support. More precisely, $\alpha$ is real and holomorphic on the boundary, equal to $1$ on $(D^{-}-(-\infty,0]\times[0,1])\cup([-\rho,2]\times[0,1])$, equal to $0$ outside $(D^{-}-(-\infty,0]\times[0,1])\cup([-\rho,\frac12\rho]\times[0,1])$, and with $|d\alpha|=\Ordo(\rho^{-1})$. Let
\[
v_{\rho}^{-}=\alpha v_{\rho}^{\widehat{\XX}}+v_{\rho}^{\partial D}+v_{\rho}^{[0,\infty)}.
\]
Similarly, let $\gamma\colon\Delta_{\rho}\to\C$ be a cut off function equal to $1$ on $(D^{+}-[0,\infty)\times[0,1])\cup([-2,\rho]\times[0,1])$, equal to $0$ outside $(D^{+}-[0,\infty)\times[0,1])\cup([-\frac12\rho,\rho]\times[0,1])$, and with $|d\gamma|=\Ordo(\rho^{-1})$. Let
\[
v_{\rho}^{+}=\gamma v_{\rho}^{\widehat{\XX}}.
\] 

Consider the elements spanned by $\widetilde{\ker}(u^{-})$ and $v_{\rho}^{-}$ as lying in $T_{u^{-}}(\XX_{1}\times\partial D\times[0,\infty))$, and those spanned by  $\widetilde{\ker}(u^{+})$ and $v_{\rho}^{+}$ as lying in $T_{u^{+}}\XX_{1,0}'$. Note that these vector fields have support in the piece of $\Delta_{\rho}$ corresponding to $D^{-}$, respectively to $D^{+}$. Using the asymptotics of the solutions of the linearized equation, compare to the discussion after Equation \ref{Eqn:VectorFieldPieces},  it is clear that for $\rho$ sufficiently large, $L^{2}$ projection gives uniformly invertible isomorphisms $\ker(u^{\pm})\to\widetilde{\ker}(u^{\pm})$, which in turn implies that the linearized Floer operator is uniformly invertible on the $L^{2}$ complement of $\widetilde{\ker}(u^{\pm})$.   

Noting that $(d_{w_{\rho}}f_{\rho})v_{\rho}^{-}$ agrees with $(D_{u^{-}}\bar\partial_{\mathrm{F}})v_{\rho}^{-}$ on its support (which lies in the part of $\Delta_{\rho}$ corresponding to $D^{-}$), and using the property $|d\alpha|=\Ordo(\rho^{-1})$, we find that
\begin{equation} \label{Eqn:ImageToZero}
\|(D_{u^{-}}\bar\partial_{\mathrm{F}}) v_{\rho}^{-}\|_{\YY_1}\to 0.
\end{equation}
Similarly, looking at the other half of $\Delta_{\rho}$, we find
\begin{equation} \label{Eqn:ImageToZero2}
\|(D_{u^{+}}\bar\partial) v_{\rho}^{+}\|_{\YY_{1}}\to 0.
\end{equation}
Since the supports of the cut off functions $\alpha$ and $\gamma$ include the supports of the elements in $\widetilde{\ker}(u^{-})$ and $\widetilde{\ker}(u^{+})$, respectively, 
\[
v_{\rho}\perp 
\left(\widetilde{\ker}(u^{-})\oplus\widetilde{\ker}(u^{+})\right) \quad  \text{implies that} \quad v_{\rho}^{-}\perp 
\widetilde{\ker}(u^{-})
\ \ \text{and}\ \ 
v_{\rho}^{+}\perp\widetilde{\ker}(u^{+}).
\] 
By the uniform invertibility of $D_{u^{-}}\bar\partial_{\mathrm{F}}$ and $D_{u^{+}}\bar\partial$ we then find from \eqref{Eqn:ImageToZero}, \eqref{Eqn:ImageToZero2} that
\begin{align*}
\|v_{\rho}^{-}\|_{N\NN}&=\|v_{\rho}^{-}\|_{T_{u^{-}}(\XX_{1}\times\partial D\times[0,\infty))}\ \to 0,\\
\|v_{\rho}^{+}\|_{N\NN}&=\|v_{\rho}^{+}\|_{T_{u^{-}}\XX_{1}'}\ \to 0.
\end{align*}
This however contradicts $\|v_{\rho}\|=1$, and we conclude that \eqref{eq:garding} holds.  In particular, $d_{w_{\rho}}f_{\rho}$ is injective.

The remaining statements of the lemma follow by observing that $d_{w_{\rho}}f_{\rho}$ comprises a smooth family of Fredholm operators of index $0$ that are injective and hence isomorphisms, and that taking the inverse is a smooth operation in this setting.
\end{proof}

\begin{Lemma}\label{Lem:partialglu2}
For $\rho_0$ sufficiently large and $\rho \geq \rho_0$, the quadratic estimate \eqref{e:QuadraticFP} for the non-linear term in the Taylor expansion of $f_{\rho}$ holds. More precisely, let $N_{w_{\rho}}$ be defined so that the following equation holds:
\[
f_{\rho}(v)=f_{\rho}(0)+(d_{w_{\rho}}f_{\rho})\,v+N_{w_{\rho}}(v),
\]
where $f_{\rho}(0)=\left(dw_{\rho}+\gamma_{\zeta,r}\otimes X_{H}(w_{\rho})\right)^{0,1}$.
Then there exists $\rho_0>0$ and a constant $C>0$ such that for all $w_{\rho}$ with $\rho>\rho_0$ and all $v,u\in N_{w_{\rho}}\NN_{\rho}$,
\begin{equation}\label{e:QuadraticNN}
\|N_{w_{\rho}}(v)-N_{w_{\rho}}(u)\|_{\widehat{\YY}_\rho}\le C\|v-u\|_{N\NN}(\|v\|_{N\NN}+\|u\|_{N\NN}).
\end{equation}
\end{Lemma}

\begin{proof}
The proof follows the lines of similar results for gluing in Floer theory. The strategy is to find a smooth function $G_{\rho}$ such that the non-linear term at $v\in N_{w_\rho}\NN$ is given by evaluating $G_{\rho}$, i.e.~$N_{w_{\rho}}(v)=G_{\rho}(w_{\rho},dw_{\rho},v,dv)$, and to show that $G_{\rho}$ as well as its first derivative with respect to $v$ vanishes at $v=0$. The estimate \eqref{e:QuadraticNN} then follows from a combination of standard Sobolev estimates.

We use
\[
E_{\rho}=\left(\Hom(T\Delta_{\rho},T\C^{n})\right)^{2}|_{\mathrm{diag}(\Delta_{\rho})}\times T(\partial D\times[0,\infty))
\]
as the source space of $G_{\rho}$, where we view the first factor as a bundle over $\Delta_{\rho}\times\Delta_{\rho}$ with fiber at $(z_1,z_2)$ equal to 
\[
(\C^{n})^{2}\times \Hom(T_{z_{1}}\Delta_{\rho},\C^{n})\times\Hom(T_{z_{2}}\Delta_{\rho},\C^{n}),
\] 
restricted to the diagonal $z_1=z_2$.

In order to find the required map $G_{\rho}\colon E_{\rho}\to \Hom^{0,1}(T\Delta_{\rho},T\C^{n})$, recall that the map $f_{\rho}$ was defined via the exponential map $\exp$ in the metric $\hat g$ followed by the Floer operator. Let $v_{0}\in T_{w_{\rho}}\widehat{\XX}_{\rho}$, let $\tau\in T_{z}\Delta_{\rho}$, and let $B_{\tau}$ be the Jacobi field along the geodesic
\[
\sigma(s)=\exp_{w_{\rho}(z)}(sv_0(z)),\quad s\in[0,1],
\]
with initial conditions
\[
B_{\tau}(0)=dw_{\rho}(\tau),\quad \nabla_{\dot\sigma(0)}B_{\tau}=\nabla_{dw_{\rho}(\tau)}v_{0}(z).
\]
Then $(d_{z}\exp_{w_{\rho}}(v))\tau=B_{\tau}(1)$. 

Write elements of $E_{\rho}$ as $e=(z,x,y,\xi,\eta,\zeta, r,\delta\zeta,\delta r)$, where $z\in\Delta_{\rho}$,  $x,y\in\C^{n}$, $\xi\in\Hom(T_{z}\Delta_{\rho},T_{x}\C^{n})$, $\eta\in\Hom(T_{z}\Delta_{\rho},T_{y}\C^{n})$, $(\zeta,\delta\zeta)\in T\partial D$, and $(r,\delta r)\in T[0,\infty)$.
Let $\tau\in T_{z}\Delta_{\rho}$ and define the bundle map $\Phi\colon E_{\rho}\to \Hom^{0,1}(T\Delta_{\rho},T\C^{n})$ as follows:
\begin{align*}
(\Phi(e))\tau &= \left(B_{\tau}(1)+\gamma_{r+\delta r,\zeta e^{i\delta\zeta}}(\tau)X_{H}(\exp_{x}(y))\right)\\ 
&+J(\exp_{x}(y))\left(B_{i\tau}(1)+\gamma_{r+\delta r,\zeta e^{i\delta\zeta}}(i\tau)X_{H}(\exp_{x}(y))\right).
\end{align*}
Then, for $v=v_{\XX}+v_{\partial D}+v_{[0,\infty)}\in T_{w_{\rho}}(\widehat{\XX}_{\rho}\times\partial D\times[0,\infty))$ we have
\[
(f_{\rho}(v))(z)=\Phi\left(z,w_{\rho},v_{\XX},dw_{\rho},dv_{\XX},\zeta,v_{\partial D},r,v_{[0,\infty)}\right).
\]

We point out for future reference that the function $\Phi$ satisfies $\frac{\partial^{2}\Phi}{\partial \eta^{2}}(e)=0$. To see this, note that the Jacobi equation is a linear ordinary differential equation and that the initial condition contains only first order terms in $\eta$. 

In order to derive the expression for the non-linear term we must study derivatives of $\Phi$. To this end we will use a trivialization of the bundle $\Hom^{0,1}(T\Delta_{\rho},T\C^{n})$ which is described in \cite{EES}. Briefly, writing $J$ for the almost complex structure on $\C^{n}$ and, for $y\in\C^{n}$,  $\Pi_y$ for parallel translation along the geodesic $\exp_{x}(sy)$, one finds $J_{y}=\Pi_{y}^{-1}J(\exp_{x}(y))\Pi_{y}$ is the almost complex structure at $\exp_{x}(y)$ transported by parallel translation to $T_{x}\C^{n}$. We trivialize the bundle of complex anti-linear maps by introducing the linear map $S_{y}=(J_{y}+J)^{-1}(J_{y}-J)$ and taking the $(J_y,i)$ complex anti linear map $A$ to the $(J,i)$ complex anti-linear map $(1-S_{y})^{-1}A$. Let $\hat\Phi$ denote $\Phi$ in that trivialization:
\[
\hat\Phi = (1-S_{y})^{-1}\Pi_{y}^{-1}\Phi.
\]
Then $\hat\Phi$ is a $(J,i)$ complex anti-linear map from $T_{z}\Delta_{\rho}$ to $T_{w_{\rho}(z)}\C^{n}$, i.e.~we have trivialized the bundle of complex anti linear maps in $E_{\rho}$ in the $y$-direction for fixed $(z,x)$. We then define the map $G_{\rho}$ in this trivialization as follows: recalling $e=(z,x,y,\xi,\eta,\zeta, r,\delta\zeta,\delta r)$,
\begin{align}\label{eq:defG}
G_{\rho}(e)&=
\hat\Phi(e) - (\xi+\gamma_{\zeta,r}\otimes X_{H}(x))^{0,1}\\\notag
&-\left(\nabla_{\xi}(y,\eta)+\left(\tfrac{\partial \gamma_{\zeta,r}}{\partial \zeta}\delta\zeta+
\tfrac{\partial\gamma_{\zeta,r}}{\partial r}\delta r\right)\otimes X_{H}(x)
+\gamma_{\zeta,r}\otimes (d_{x}X_{H})y\right)^{0,1}\\\notag 
&-\tfrac12J\nabla_yJ\left(\xi+\gamma_{\zeta,r}\otimes X_{H}(x)\right)^{1,0},
\end{align}
where, for $\tau\in T_{z}\Delta_{\rho}$,  $\nabla_{\xi(\tau)}(y,\eta(\tau))$ denotes the covariant derivative in direction $\xi(\tau)$ of a vector field equal to $y$ at $x$ and with first derivative $\eta(\tau)$.

Noting that the last two lines in \eqref{eq:defG} correspond to the linearization of $f_{\rho}$, compare the calculation in ~\cite[Proof of Lemma 3.5, p.~3319]{EES}, we find that 
\[
(N_{w_\rho}(v))(z)=G_{\rho}(z,w_{\rho},v_{\widehat\XX},dw_{\rho},dv_{\widehat\XX},\zeta,r,v_{\partial D},v_{[0,\infty)}).
\]

We next study $G_{\rho}$ for small $\epsilon y$ when $\epsilon\to 0$. A straightforward calculation using the first order equation
\[
S_{\epsilon y} =  -\epsilon\frac12 J\nabla_{y}J +\Ordo(\epsilon^{2}),
\] 
see \cite[Equation (3.13)]{EES} then shows that
\[
G(z,x,\epsilon y,\xi,\epsilon\eta,\zeta,r,\epsilon(\delta\zeta),\epsilon(\delta r))=\Ordo(\epsilon^{2})
\]
and we find that
\[
G(z,x,0,\xi,0,\zeta,r,0,0)=0\quad\text{ and }\quad d_{y,\eta,\delta\zeta,\delta r}G(z,x,0,\xi,0,\zeta,r,0,0)=0.
\]

Consider now 
\begin{align*}
&N_{w_\rho}(v)-N_{w_\rho}(u)=\\
&G_{\rho}(z,w_{\rho},v_{\widehat\XX},dw_{\rho},dv_{\widehat\XX},\zeta,r,v_{\partial D},v_{[0,\infty)})-
G_{\rho}(z,w_{\rho},u_{\widehat\XX},dw_{\rho},du_{\widehat\XX},\zeta,r,u_{\partial D},u_{[0,\infty)}).
\end{align*}

After addition and subtraction of suitable terms we need only consider inputs that differ in one argument. We focus on the nontrivial, infinite dimensional part of the estimate and suppress the Hamiltonian components as well as $z$ and $w_{\rho}$ from the notation. Then we have
\begin{align*}
|G_{\rho}(v,dv)-G_{\rho}(u,du)|&\le|dG_{\rho}|(|u-v|+|dv-du|)\\
&\le C(|v|+|u|+|du|+|dv|)(|u-v|+|dv-du|),
\end{align*}
where we estimate $dG$ using the fact that it vanishes at $0$. To get the desired estimate, we square and integrate. Using the fact that the $\sblv^{1}$-norm (i.e.~one derivative in $L^{2}$) controls the $L^{4}$-norm we conclude that the $L^{2}$-norm of the left hand side is dominated by
\[
C(\|u\|_{N\NN}+\|v\|_{N\NN})\|u-v\|_{N\NN}.
\]  

We must estimate one more derivative:
\begin{align*}
d(G_{\rho}(v,dv)-G_{\rho}(u,du))&=D_1G_{\rho}(v,dv)dv + D_2G_{\rho}(v,dv)d^{2}v\\ 
&-D_1G_{\rho}(u,du)du -D_2G_{\rho}(u,du)d^{2}u.
\end{align*}
For the difference of the first and the third terms we find 
\begin{align*}
|D_1G_{\rho}(v,dv)dv-D_1G_{\rho}(u,du)du|&\le |D_1G_{\rho}(v,dv)dv-D_1G_{\rho}(u,du)dv|
+|D_{1}G_{\rho}(u,du)(dv-du)|\\
&\le C((|v-u|+|dv-du|)|dv|+(|u|+|du|)|dv-du|)\\
&\le C(|v|+|u|+|du|+|dv|)(|u-v|+|dv-du|),
\end{align*}
which gives the desired estimate exactly as above.

To estimate the difference of the second and fourth terms we use the fact that $G_{\rho}$ has first order dependence on $\eta$ (this was discussed for $\Phi$ above, and is obvious for the term subtracted from it in defining $G_{\rho}$). In the present language this means that $D_{2}G_{\rho}(y,\eta)$ is in fact independent of $\eta$, $D_{2}G_{\rho}(y,\eta)=D_{2}(y)$, and we get
\begin{align*}
|D_2G_{\rho}(v)d^{2}v-D_2G_{\rho}(u)d^{2}u|&\le |(D_{2}G_{\rho}(v)-D_{2}G_{\rho}(u))d^{2}v|+|D_{2}G_{\rho}(u)(d^{2}v-d^{2}u)|\\
&\le C(|u-v||d^{2}v|+|u||d^{2}u-d^{2}v|)\\
&\le C((|u|+|v|)|d^{2}u-d^{2}v|+(|d^{2}u|+|d^{2}v|)|u-v|).
\end{align*}  
Squaring and integrating, using the fact that the $\sblv^{2}$-norm (two derivatives in $L^{2}$) controls the supremum norm, we find that this term is also controlled by
\[
C(\|u\|_{N\NN}+\|v\|_{N\NN})\|u-v\|_{N\NN}.
\]  
This finishes the proof.
\end{proof}

Lemmas \ref{Lem:almostholomorphic}, \ref{Lem:partialglu1}, and \ref{Lem:partialglu2} have the following consequence:

\begin{Corollary} \label{Cor:C1Structure}
The Newton iteration map with initial values in $\NN_{\rho}=\Pre(\NN_\rho)\subset N\NN_{\rho}$ gives a $C^{1}$-diffeomorphism
\[
\Phi\colon \FF_1\times\MM\times[\rho_0,\infty)\to f^{-1}(0),
\]
where $0$ denotes the $0$-section in $\widehat{\tcfig}$,
such that $\Phi(u_1,u_2,\rho)$ limits to a broken curve with components $u_1\in\FF_1$ and $u_2\in\MM$ as $\rho\to\infty$.
\end{Corollary}

\begin{proof}
This is now a direct consequence of Lemma \ref{Lem:FloerPicard}.
\end{proof}

\subsection{A smooth neighborhood of the Gromov Floer boundary} \label{Ssec:NhoodBdary}
In this section we show that $\exp(f^{-1}(0))\subset\widehat{\cfig}\times[0,\infty)$, which is diffeomorphic to $\NN$ via the gluing map $\Phi$, gives a $C^{1}$-collar neighborhood of the Gromov-Floer boundary of $\FF_{0}$. The argument presented here is parallel to that of \cite[Section 6.6]{EkholmSmith}.  

Fix $\rho\in [\rho_0,\infty)$, and let $f_{\rho}^{-1}(0)=f^{-1}(0)\cap N\NN_{\rho}$. 
We introduce a re-parametrization map, 
\[
\Psi_{\rho}\colon f_{\rho}^{-1}(0)\to\FF_0,
\]  
as follows. Write the domain $D$ as $H_{\rho}$, see Remark \ref{rem:H_rho}, where the region in $\Delta_{\rho}$ corresponding to $D^{-}$ corresponds to the upper half plane with a small disk at the origin removed, and where the region in $\Delta_{\rho}$ corresponding to $D^{+}$ corresponds to the remaining small disk and has been scaled by $\delta e^{-2\pi\rho}$.  An element  $u\in N\NN_{\rho}$ gives a map $\Psi_{\rho}(u)\colon D\to\C^{n}$, which is the re-parametrization discussed above of the map $\exp_{w_{\rho}}(u)$, originally defined on $\Delta_{\rho}$. The map $\Psi_{\rho}(u)$ then solves the Floer equation on $D$ if and only if $f_{\rho}(u)=0$.  We let $\Psi\colon f^{-1}(0)\to \FF_0$ be the map which equals $\Psi_{\rho}$ on $f_{\rho}^{-1}(0)$.

By definition of the pre-gluing map, it follows that as $\rho\to\infty$, 
\[
\inf_{u\in f_{\rho}^{-1}(0)}\left\{ \sup_{D}|d\Psi(u)|\right\}\to\infty.
\]
In particular for any $M>0$, there is $\rho_0>0$ such that $\Psi(u)\in U_M$ for all $u\in N\NN_{\rho}$ and all $\rho>\rho_0$.

\begin{Lemma}\label{Lem:gluemb}
For $\rho$ sufficiently large, the map $\Psi$ is a $C^{1}$ embedding.
\end{Lemma}

\begin{proof}
The distance functions in $\FF_{0}$ and $\NN$ are induced from configuration spaces which are Banach manifolds with norms that control the $C^{0}$-norm. Since all norms are equivalent to the $C^{0}$-norm for (Floer) holomorphic maps, we may think of all distances $\mathrm{d}(\cdot,\cdot)$ in the calculations below as $C^{0}$-norms. Equation \eqref{Eq:iterest} and the rescaling of the part of $H_{\rho}$ that corresponds to $D^{+}$ imply that 
\begin{equation} \label{Eq:DerivativeBound}
\left|\frac{\pa(\Psi_{\rho}(u))}{\pa\rho}\right|_{C^{0}}\ge C e^{\pi\rho}
\end{equation}
for some constant $C>0$. 

Assume now that that there is a sequence of pairs $(u,\rho)\ne (u',\rho')$ in $\NN=(\FF_{1}\times\MM)\times[\rho_0,\infty)$ such that 
\[
\mathrm{d}(\Psi_{\rho}(u),\Psi_{\rho'}(u'))=\Ordo(\mathrm{d}(u,u')).
\]
Then it would follow that
\begin{equation} \label{Eq:InjectiveGlue}
\mathrm{d}(\Psi_{\rho}(u),\Psi_{\rho'}(u))\le \mathrm{d}(\Psi_{\rho'}(u),\Psi_{\rho'}(u'))+\Ordo(\mathrm{d}(u,u'))=\Ordo(\mathrm{d}(u,u')).
\end{equation}
However, taking $\epsilon_0>0$ small and $\mathrm{d}(u',u)\le \epsilon_0$, \eqref{Eq:InjectiveGlue} contradicts the derivative bound \eqref{Eq:DerivativeBound} once $\rho, \rho'$  are sufficiently large. We conclude that for small enough $\epsilon_0$,  the map is a $C^{1}$ embedding in an $\epsilon_0$-neighborhood of any point.

It follows that the map is also a global embedding: since $\Psi$ approaches the pregluing map as $\rho\to \infty$, there is $\rho_0>0$ so that for $\rho, \rho' > \rho_0$, if  $\mathrm{d}(u,u')>\epsilon_0$ then $\mathrm{d}(\Psi_{\rho}(u),\Psi_{\rho'}(u'))\ge \frac12\epsilon_0$.
\end{proof}

We are now ready to prove our main gluing result.

\begin{Theorem}\label{Thm:gluing}
For $\rho_0$ sufficiently large, the map $\Psi\colon \NN \rightarrow \FF_0$ is a smooth embedding onto a neighborhood of the Gromov-Floer boundary.
\end{Theorem}

\begin{proof}
Take $\rho_0$ sufficiently large that the conclusion of Lemma \ref{Lem:gluemb} holds. 
It only then remains to show that (perhaps for some larger $\rho_0$) the map $\Psi$ is surjective onto a neighborhood of the Gromov-Floer boundary of $\FF_0$. 

To see this we consider a Gromov-Floer convergent sequence $u_{j}$ of solutions in $\FF_{0}$ with non-trivial bubbling. Such a sequence converges uniformly on compact subsets, and by Corollary \ref{Cor:limit2} we recover $u^{-}\in\FF_{1}$ and $u^{+}\in\MM$ as maps corresponding to these uniform limits. We need then to show that $u_j$, when re-parameterized and considered as a map on $H_{\rho}\approx\Delta_{\rho}$, eventually lies in the image of the exponential map of a small neighborhood of $u^{-}\#_{\rho} u^{+}$. We point out that the inverse of the re-parameterization map determines the relevant value of $\rho$ from the distance $d$ of the intersection points with $S_{\epsilon}(a^{\pm})$ corresponding to the marked points in $D^{+}$. If, as in Remark \ref{rem:H_rho}, we  think of $D^{+}$ as the upper half plane with marked points at $\pm 1$, then $\rho=(2\pi)^{-1}\log(d/2\delta)$.

We claim that $u_j$ eventually lies in the image under the exponential map of a small neighborhood of $\NN \subset N\NN$ where the Newton iteration map is defined. Corollary \ref{Cor:limit2} implies for any fixed $\rho_0$ and for $k\leq 2$, the maps $u_j$ $C^{k}$-converge  on the two ends of the strip region $[-\rho+\rho_0,\rho-\rho_0]\times[0,1]$ to $u^{-}$ and $u^{+}$, respectively. On the remaining growing strip we have a holomorphic map which converges to a constant (since there cannot be further bubbling). 
In particular, the map of the strip eventually lies completely inside the region where the Lagrangian immersion agrees with $\R^{n}\cup i\R^{n}$. Since the almost complex structure is standard here and since the Hamiltonian term vanishes in the region in $H_{\rho}$ corresponding to the strip, we can write down the
solutions explicitly:
\[
u_j(z)=\sum_{k\in\Z} c_ke^{\pi\left(k-\tfrac12\right)z},
\]
where $c_k\in\R^{n}$. Writing $u_j=x_j+iy_j$, with $x_j,y_j\in\R^{n}$, and $z=\tau+it\in[-\rho',\rho']\times[0,1]$, $\rho'=\rho-\rho_0$, the $C^{0}$ convergence together with the explicit form of the solutions $u^{\pm}$ near the punctures, see Remark \ref{rem:Fourierexpansion}, then implies that
\[
\int_{0}^{1}\left\langle y_j(\pm\rho',t),y_j(\pm\rho',t)\right\rangle \,dt=\Ordo(e^{-\pi\rho_0}).
\]
On the other hand, 
\begin{align*}
&\int_{0}^{1}\left\langle y_j(\rho',t),y_j(\rho',t)\right\rangle \,dt+
\int_{0}^{1}\left\langle y_j(-\rho',t),y_j(-\rho',t)\right\rangle \,dt=\\
&\sum_{k\in\Z} |c_k|^{2}(e^{(2k-1)\pi\rho'}+e^{-(2k-1)\pi\rho'}).
\end{align*}
and
\begin{align*}
&\int_{-\rho'}^{\rho'}\int_{0}^{1}(\left\langle x_j(\rho,t),x_j(\rho,t)\right\rangle+\left\langle y_j(\rho,t),y_j(\rho,t)\right\rangle) \,dtd\rho=\\
&\sum_{k\in\Z} \frac{|c_k|^{2}}{|2k-1|\pi}\left|e^{(2k-1)\pi\rho'}-e^{-(2k-1)\pi\rho'}\right|.
\end{align*}

We thus find that the $C^{0}$ norm at the ends of the strip controls the $L^{2}$-norm. Repeating this argument for two derivatives of $u_j$ we see that the relevant Sobolev norm is controlled by the $C^{2}$-convergence at the ends of the strip. We conclude that the re-parameterized version of $u_j$ eventually lies in the image of a small neighborhood of $\NN$ in $N\NN$. Thus, the gluing map is a surjective embedding onto a neighborhood of the Gromov Floer boundary of $\FF_0$. 
\end{proof}

As noted in Section \ref{Ssec:Overview}, the proof of Theorem \ref{Thm:C1boundary} is now completed by gluing together (the $C^1$-smooth pieces) $\FF_0 - \overline{U_M}$ and $\FF_1 \times \MM \times [\rho_0,\rho_1]$, for suitably large $M, \rho_1$, via the diffeomorphism $\Psi$ constructed above. 

\bibliographystyle{alpha}

\end{document}